\documentclass[preprint,12pt]{elsarticle}
\usepackage{epsfig}
\usepackage{eucal}
\usepackage{amssymb,amsmath,epstopdf,amsfonts}
\usepackage{latexsym}
\usepackage{amsfonts}
\usepackage{tikz}
\usepackage{graphicx}
\usepackage{array}
\usepackage[utf8]{inputenc}
\usepackage{hyperref}
\usepackage{comment}
\usepackage{booktabs}
\usepackage{float}
\usepackage{siunitx}
\usepackage{subcaption}
\hypersetup{colorlinks,linkcolor={blue},citecolor={blue},urlcolor={red}}  
\usepackage{graphicx}%
\usepackage{multirow}%
\usepackage{amsthm}%
\usepackage{mathrsfs}%
\usepackage[title]{appendix}%
\usepackage{xcolor}%
\usepackage{textcomp}%
\usepackage{manyfoot}%
\usepackage{booktabs}%
\usepackage{algorithm}%
\usepackage{algorithmicx}%
\usepackage{algpseudocode}%
\usepackage{listings}%
\usepackage[english]{babel}
\usepackage[nottoc]{tocbibind}
\newtheorem{thm}{Theorem}
\newtheorem{remark}{Remark}%
\newtheorem{lemma}{Lemma}
\begin{document}

\begin{frontmatter}
\title{Distributed Delay Systems with Oscillations\\in the Time Histories}
%\maketitle
%	\author{Adam Mielke}\email{admi@dtu.dk}\affiliation{Dynamical Systems, Department of Applied Mathematics and Computer Science, Technical University of Denmark, Asmussens %Allé, 303B, 2800 Kgs.\ Lyngby, Denmark}
%	\author{Mads Peter Sørensen}\email{mpso@dtu.dk}\affiliation{Dynamical Systems, Department of Applied Mathematics and Computer Science, Technical University of Denmark, %Asmussens Allé, 303B, 2800 Kgs.\ Lyngby, Denmark}
%	\author{John Wyller}\email{john.wyller@nmbu.no}\affiliation{Department of Mathematics, Faculty of Science and Technology, Norwegian University of Life Sciences, P.O. Box 5003, %NO-1432 Ås, Norway}

\author[dtu]{Adam Mielke\corref{cor1}}
\ead{kontakt@mielke.dk}
\author[dtu]{Mads Peter Sørensen}
\ead{mpso@dtu.dk}
\author[nmbu]{John Wyller}
\ead{john.wyller@nmbu.no}
\cortext[cor1]{Corresponding author}
\address[dtu]{Department of Applied Mathematics and Computer Science, Technical University of Denmark, Richard Petersens Plads, Bldg. 324, DK-2800 Kgs.\ Lyngby, Denmark}
\address[nmbu]{Department of Mathematics, Faculty of Science and Technology, Norwegian University of Life Science, P.O. Box 5003, NO-1432, \AA s, Norway.}

\begin{abstract}
	
We design a Linear Chain Trick (LCT)-algorithm for dynamical systems with distributed time delay where the time histories contain temporal oscillations. The methodology is illustrated  by means of an example in population dynamics. 

\end{abstract}

\end{frontmatter}
%\maketitle	

\section{Introduction}\label{Introduction}

 The Linear Chain Trick (LCT) has emerged as a tool for studying dynamical systems with distributed time delay. See \cite{cushing2013,macdonald1978difference} and the references therein. It represents an efficient methodology to convert a class of distributed delay systems to systems of ordinary differential equations (ODEs). The latter type of systems can easily be analysed by means of standard techniques and numerical computations.\\
\newline
 The LCT is also generalized to a method which can be used to construct mean field ODE models from continuous-time stochastic state transition models where the time an individual spends in a given state (i.e., the dwell time) is Erlang distributed (i.e., gamma distributed with integer shape parameter). See \cite{hurtado2019generalizations,hurtado2021building}.\\
\newline
Distributed delay systems appear as descriptions for genetic networks \cite{shlykova2008general,shlykova2010singular}, epidemiological processes \cite{MielkeSorensenWyller}, predator-prey interactions, bio-economical models \cite{Bergland2022TimeDA} and economic growth models (of the Kaldor-Kalecki type) \cite{guerrini2020bifurcations}. Moreover, singular perturbation theory for dynamical systems offers also a motivation for studying general properties of modified versions of LCT.  This was a key issue in \cite{shlykova2008general} and \cite{shlykova2010singular}. A delay estimation procedure in distributed delay system is proposed in \cite{RitschelWyller}.  Last but not least, distributed delay systems appear in non-local, nonlinear wave packet models in optics with the Raman response incorporated \cite{blow1989theoretical,laegsgaard2007mode,dudley2010supercontinuum}.\\
\newline
This means that the LCT-technique has a wide range of applications, which serves as a background and motivation for exploring different aspects of the LCT-procedure further.
In this paper, we look at the role of oscillations in the  dependence on the history. That is, if the integration kernel itself is oscillating. This comes about when a particular time scale is at play, such as remembering how the world was last time it was this particular weekday or this particular season.\\
\newline
 The present paper is organized as follows: Section \ref{Examples} is devoted to a motivating example from population dynamics illustrating the application of the LCT-technique to distributed time delay systems with oscillations in the time history. Here the oscillations appear as perturbations on top of Erlang distributed temporal kernels.
 The methodology used in this example extends the methodology presented in \cite{cushing2013} and \cite{ponosov2004thew}. Here the focus is on the impact of the oscillatory time history on local structural stability around hyperbolic equilibrium points and Hopf-bifurcations. In Section \ref{LCT} we present the LCT-method for a general dynamical system with oscillations in the time history where the temporal oscillations are perturbations superimposed on top of Erlang distributed memory functions. Section \ref{Concluding remarks and extensions} contains concluding remarks and a discussion of possible extensions. In \ref{Continuous dependence on the time history} we discuss the robustness property of the LCT-algorithm: We prove that the mapping from the Banach space of absolute integrable memory functions to the Banach space of absolute continuous solutions of the distributed time delay system is continuous with respect to the metrics (Theorem \ref{stabilitytheorem}). This theorem sheds light on the dynamical properties of the ODE-systems which we derive in Section \ref{Examples} by means of the LCT-algorithm.

\section{Example and Construction of Oscillations in Time History}\label{Examples}

In this section we demonstrate the application of the LCT-methodology to distributed time delay systems with oscillations in the time history to a logistic equation. We calculate this example explicitly, as it will be used as motivation for the general framework we introduce in Section \ref{LCT}.\\
\newline
%Here we will show examples on how to extend the LCT - methodology in \cite{cushing2013} and \cite{ponosov2004thew} to cases with temporal oscillations in the time delay.  We have incorporated distributed delay with oscillations in the logistic equation (Subsection \ref{Logistic equation with oscillatory distributed time delays}), the economic growth model of the Kaldor - Kalecki type (Subsection \ref{Economic growth model})  the SIR - and the SEIR - model  (Subsubsection \ref{SIR - model with oscillatory distributed time delays} and Subsubsection \ref{SEIR - model with oscillatory distributed time delays}) and Raman scattering processes in nonlinear optics (Subsubsection \ref{Raman scattering processes in nonlinear optics}). The purpose of this study is to detail the effect of oscillations in the memory effects and to  illustrate the robustness property which is  summarized in Theorem \ref{stabilitytheorem}: The mapping from space of absolute integrable memory functions to the space of absolute continuous solutions of the distributed time delay system is continuous with respect to the metrics.
The standard logistic equation
\begin{equation}
\dot{x}=r x\left(1-\frac{x}{K}\right)\label{logisticbasic}
\end{equation}
serves as a simple skeleton model for population growth in an ecological or bioeconomical system where limited food resources cause a saturation in this growth. Here $x$ is the population density, $r>0$ the logistic growth rate and  $K$ the constant carrying capacity.\\
\newline
 We extend the model (\ref{logisticbasic}) to the distributed delay system  
\begin{equation}
    \dot {x}= r x\left(1 - \frac{\alpha\otimes x}{K}\right)\label{logisticdistributeddelay}
\end{equation}
where $\alpha\otimes x$ is the convolution integral defined as 
\begin{align}
   [\alpha\otimes x](t) &\equiv \int\limits_{-\infty}^{t} \alpha(t - s) x(s)\,\mathrm ds,
\end{align}
%
%and $K: \mathbb{R}\rightarrow (0, \infty)$ is the time-varying carrying capacity of the population given by
%
%\begin{align}
  %  K(t) &= \left(1 + A_1 \sin(2\pi \omega_1 t) + A_2 \sin(2\pi \omega_2 t)\right) \bar K,
%\end{align}
%
%where $A_1, A_2 \in [0, \infty)$ are amplitudes, $\omega_1, \omega_2 \in [0, \infty)$ are frequencies, and $\bar K \in (0, \infty)$ is a nominal carrying capacity. Here we tacitly assume that the integral kernel $\alpha:[0,+\infty)\rightarrow (0, \infty)$ is absolute integrable.\\
%\newline
It is tacitly assumed that the integral kernel $\alpha$ is absolute integrable.\\
\newline
Let us model the memory effect by means of the perturbed Erlang distribution
\begin{equation}
\label{system3logistic}
\alpha(t)=\theta(t)+\frac{1}{2}\varepsilon\big[\varphi(t)+(\varphi)^{\ast}(t)\big]
\end{equation}
where $\theta$ and  $\varphi$ are defined as
\begin{equation}
\label{system4logistic}
\left.\begin{aligned}
&\theta(t)=\sigma^{2}t e^{-\sigma t}&\\
&\varphi(t)=\sigma^{2}t e^{\lambda t},\quad \lambda\equiv-\sigma+i\Omega, \quad (i^{2}=-1)&
\end{aligned} \right\}
\end{equation}
and the operation $\ast$ means complex conjugation. This means that we account for finite constant amplitude oscillations (with the frequency $\Omega$) superimposed on the background 'strong'  memory kernel. We proceed as follows: Introduce the auxilliary variables $V_{1}$, $V_{2}$, $W_{1}$ and $W_{2}$ corresponding to these  kernels, i.e., 
\begin{subequations}
\label{auxilliaryvariables}
\begin{eqnarray}
&V_{1}(t)= [\theta_{0}\otimes x](t)\equiv\int\limits_{-\infty}^{t}\theta_{0}(t-s)x(s)ds,\quad \theta_{0}(t)=\sigma e^{-\sigma t}&\label{auxilliaryvariablesa}\\
&&\nonumber\\
&V_{2}(t)= [\theta\otimes x](t)\equiv\int\limits_{-\infty}^{t}\theta(t-s)x(s)ds&\label{auxilliaryvariablesb}\\
&&\nonumber\\
&W_{1}(t)= [\phi_{0}\otimes x](t)\equiv\int\limits_{-\infty}^{t}\phi_{0}(t-s)x(s)ds,\quad \phi_{0}(t)=\sigma e^{\lambda t}&\label{auxilliaryvariablesc}\\
&&\nonumber\\
&W_{2}(t)=[\phi\otimes x](t)\equiv\int\limits_{-\infty}^{t}\phi(t-s)x(s)ds&\label{auxilliaryvariablesd}
\end{eqnarray}
\end{subequations}
and derive the rate equations 
\begin{subequations}
\label{rate1}
\begin{eqnarray}
&\dot{V}_{1}=-\sigma V_{1}+\sigma x,\quad \dot{V}_{2}=-\sigma V_{2}+\sigma V_{1}&\label{rate1a}\\
&&\nonumber\\
&\dot{W}_{1}=\lambda W_{1}+\sigma x,\quad \dot{W}_{2}=\lambda W_{2}+\sigma W_{1}&\label{rate1b}
\end{eqnarray}
\end{subequations}
for $V_{k}$, $W_{k}$, (k=1,2). The rate equations for $W_{1}=u_{1}+iv_{1}$ and $W_{2}=u_{2}+iv_{2}$ (with $u_{k}, v_{k}\in\mathbb{R}$) decompose into real and imaginary parts, i.e., 
\begin{subequations}
\label{rate2}
\begin{eqnarray}
&\dot{u}_{1}=-\sigma u_{1}-\Omega v_{1}+\sigma x&\label{rate2a}\\
&&\nonumber\\
&\dot{v}_{1}=-\sigma v_{1}+\Omega u_{1}&\label{rate2b}\\
&&\nonumber\\
&\dot{u}_{2}=-\sigma u_{2}-\Omega v_{2}+\sigma u_{1}&\label{rate2c}\\
&&\nonumber\\
&\dot{v}_{2}=-\sigma v_{2}+\Omega u_{2}+\sigma v_{1}&\label{rate2d}
\end{eqnarray}
\end{subequations}
The dynamical equation for $x$ then assumes the form
\begin{equation}
  \dot {x}= rx\left(1 - \frac{V_{2}+\varepsilon u_{2}}{K}\right).\label{logisticdelay}
\end{equation}
Hence the actual distributed time delay equation transforms to a $7D$ coupled autonomous system of ODEs, namely the equations (\ref{rate1a}), (\ref{rate2}) and (\ref{logisticdelay}).   
%Notice that we can transform this system to an autonomous system defined on the extended phase space $\mathbb{R}^{11}=\{(N,V_{1},V_{2},u_{1},v_{1},u_{2},v_{2},a_{1},b_{1},a_{2},b_{2})\}$ where we have introduced the additional auxiliary state variables $a_{1}$, $b_{1}$, $a_{2}$ and $b_{2}$. The evolution of these state variables is prescribed by means of the stadard linear oscillator problem
%\begin{subequations}
%\label{oscillatorproblemcarryingcapacity}
%\begin{eqnarray}
%&\dot{a}_{1}(t)=b_{1}(t),\quad a_{1}(0)=0&\label{oscillatorproblemcarryingcapacitya}\\
%&&\nonumber\\
%&\dot{b}_{1}(t)=-(2\pi\omega_{1})^{2}a_{1}(t),\quad b_{1}(0)=2\pi\omega_{1}A_{1}\label{oscillatorproblemcarryingcapacityb}\\
%&&\nonumber\\
%&\dot{a}_{2}(t)=b_{2}(t),\quad a_{2}(0)=0&\label{oscillatorproblemcarryingcapacityc}\\
%&&\nonumber\\
%&\dot{b}_{2}(t)=-(2\pi\omega_{2})^{2}a_{2}(t),\quad b_{2}(0)=2\pi\omega_{2}A_{2}&\label{oscillatorproblemcarryingcapacityd}
%\end{eqnarray}
%\end{subequations}
%and the carrying capacity $K$ can now be expressed in terms of $a_{1}$ and $a_{2}$ as
%\begin{equation}
%K(t)=\left(1 + a_1 (t) + a_2 (t)\right) \bar K\label{carryingcap}
%\end{equation}
We conveniently write the resulting autonomous system on the compact vector form
\begin{equation}
\dot{\mathbf{X}}=\mathbf{F}(\mathbf{X})\label{Gsystem1}
\end{equation}
where 
\begin{equation}
\mathbf{X}=[x,V_{1},V_{2},u_{1},v_{1},u_{2},v_{2}]\label{Gsystem2}
\end{equation}
%and $\mathbf{k}$ denotes a vector that contains all the parameters involved. 
The vector field $\mathbf{F}:\mathbb{R}^{7}\rightarrow\mathbb{R}^{7}$ is defined as
\begin{equation}
\mathbf{F}(\mathbf{X})=[F_{1}(\mathbf{X}), F_{2}(\mathbf{X}),\cdots,F_{7}(\mathbf{X})]\label{Gsystem3}
\end{equation}
where the components functions $F_{j}$, $(j=1,2,\cdots,7)$ are given as
\begin{subequations}
\label{Componentfunctionsvectorfieldlogistic}
\begin{eqnarray}
&F_{1}(\mathbf{X})=r x(1-\frac{V_{2}+\varepsilon u_{2}}{K})&\label{Componentfunctionsvectorfieldlogistica}\\
&&\nonumber\\
&F_{2}(\mathbf{X})=-\sigma V_{1}+\sigma x,\quad F_{3}(\mathbf{X})=-\sigma V_{2}+\sigma V_{1}&\label{Componentfunctionsvectorfieldlogisticb}\\
&&\nonumber\\
&F_{4}(\mathbf{X})=-\sigma u_{1}-\Omega v_{1}+\sigma x,\quad F_{5}(\mathbf{X})=-\sigma v_{1}+\Omega u_{1}&\label{Componentfunctionsvectorfieldlogisticc}\\
&&\nonumber\\
&F_{6}(\mathbf{X})=-\sigma u_{2}-\Omega v_{2}+\sigma u_{1},\quad F_{7}(\mathbf{X})=-\sigma v_{2}+\Omega u_{2}+\sigma v_{1}&
\label{Componentfunctionsvectorfieldlogisticd}
%&&\nonumber\\
%&H_{8}(\mathbf{Z})=b_{1},\quad H_{9}(\mathbf{Z})=-(2\pi\omega_{1})^{2}a_{1}&\label{Componentfunctionsvectorfieldlogistice}\\
%&&\nonumber\\
%& H_{10}(\mathbf{Z})=b_{2},\quad H_{11}(\mathbf{Z})=-(2\pi\omega_{2})^{2}a_{2}&\label{Componentfunctionsvectorfieldlogisticef}
\end{eqnarray}
\end{subequations}
The autonomous dynamical system (\ref{Gsystem1}) - (\ref{Componentfunctionsvectorfieldlogistic}) can now be analyzed with respect to existence and linear stability of equilibrium states in the standard way.\\
\newline
%In what follows we will restrict ourselves to the special case when $A_{1}=A_{2}=0$ in the carrying capacity function $K$, which means that the carrying capacity is constant in time. i.e.,  $K(t)=\bar K$.\\
%\newline
We start out with a summary of the stability properties when $\varepsilon=0$. As expected the stability analysis then relies on the $3D$ - subsystem consisting of the equations (\ref{rate1a}) and  (\ref{logisticdelay}). We observe that $Q_{0}: [x,V_{1},V_{2}]=[0,0,0]$  and $Q_{e}: [x,V_{1},V_{2}]=[K,K,K]$ are the equilibrium points of this subsystem.  For the trivial equilibrium $Q_{0}$, the starting point for the stability analysis is the Jacobian
 \begin{equation}
\mathbf{J}(Q_{0})
=\left[
  \begin{array}{cccccccc}
r& 0&0& \\
\sigma& -\sigma&0\\
0&\sigma&-\sigma
    \end{array}
\right]\label{matrixP0}
\end{equation}
from which we get the characteristic polynomial
\begin{equation}
\mathcal{P}_{0}(z)=(z-r)(z+\sigma)^{2}\label{P0}
\end{equation}
That is, the eigenvalues are $r$ and $-\sigma$.
Because $r>0$, we conclude that the trivial equilibrium point $P_{0}=[0,0,0]$ is always unstable. The nontrivial equilibrium of that subsystem (\ref{rate1a}) and  (\ref{logisticdelay}) is given by $Q_{e}:[x,V_{1},V_{2}]=[K,K,K]$. We readily find that the Jacobian of the vector field defining this subsystem evaluated at $Q_{e}$ is given as
\begin{equation}
\mathbb{J}_{33}=\left[
  \begin{array}{cccccccc}
0& 0& -r\\
\sigma& -\sigma&0\\
0&\sigma&-\sigma
    \end{array}
\right]\label{Jacobian3D}
\end{equation}
The characteristic polynomial $\mathcal{P}_{3}$ of $\mathbb{J}_{3}$  becomes
\begin{eqnarray}
& \mathcal{P}_{3}(z)\equiv\vert z\mathbb{I}_{33}-\mathbb{J}_{33}\vert=z^{3}+a_{1}z^{2}+a_{2}z+a_{3}&\nonumber\\
&&\label{characteristicpol3}\\
&a_{1}=2\sigma,\quad a_{2}=\sigma^{2},\quad a_{3}=r\sigma^{2}&\nonumber
\end{eqnarray}
Here $\mathbb{I}_{33}$ is the $3\times 3$-unit matrix. The Routh-Hurwitz matrices are given as \cite{hurwitz1964conditions}
\begin{eqnarray}
&\mathbf{D}_{1}=[a_{1}]=[2\sigma],\quad \mathbf{D}_{2}=\left[
  \begin{array}{cccccccc}
a_{1}& a_{3}\\
1& a_{2}
    \end{array}
\right] =\left[
  \begin{array}{cccccccc}
2\sigma&r\sigma^{2}\\
1& \sigma^{2}
    \end{array}
\right]&\nonumber\\
&&\nonumber\\
&\mathbf{D}_{3}=\left[
  \begin{array}{cccccccc}
a_{1}& a_{3}&0\\
1& a_{2}&0\\
0& a_{1}&a_{3}
    \end{array}
\right]=\left[
  \begin{array}{cccccccc}
2\sigma&r\sigma^{2}&0\\
1&\sigma^{2}&0\\
0& 2\sigma&r\sigma^{2}
    \end{array}
\right]&\nonumber
\end{eqnarray}
from which the expressions 
\begin{equation}
\vert\mathbf{D}_{1}\vert=2\sigma, \quad \vert\mathbf{D}_{2}\vert=2\sigma^{2}(\sigma-\frac{1}{2}r),\quad \vert\mathbf{D}_{3}\vert=2r\sigma^{4}(\sigma-\frac{1}{2}r)\nonumber
\end{equation}
for the Routh-Hurwitz determinants follow. The Routh-Hurwitz  criterion implies that all the zeros of $\mathcal{P}_{3}$ are located in left half $z$-plane if and only if $\vert\mathbf{D}_{j}\vert>0$ for $j=1,2,3$. The latter condition is fulfilled if and only if $\sigma>\frac{1}{2}r$. We therefore conclude that the equilibrium point $[K,K,K]$  is asymptotically stable if $\sigma>\frac{1}{2}r$ and unstable in the complementary regime $0<\sigma<\frac{1}{2}r$. Notice that this result agrees with stability result obtained in \cite{cushing2013} by direct linearization of the model (\ref{logisticdistributeddelay}) about the equilibrium $x=K$. We readily find that the transversality condition
\begin{equation}
\vert\mathbf{D}_{2}\vert=0\Leftrightarrow \sigma=\frac{1}{2}r,\quad \partial_{\sigma}\big[\vert\mathbf{D}_{2}\vert\big](\sigma=\frac{1}{2}r)=\frac{1}{2}r^{2}\neq 0\nonumber
\end{equation}
is satisfied. Notice that the eigenvalues of the Jacobian $\mathbf{J}_{3}$ are given as
\begin{equation}
z_{1}=-r < 0,\quad z_{\pm}=\pm i\sqrt{a_{3}/a_{1}} = \frac{i}{2}r\label{eigenvaluesJ3}
\end{equation}
when $\sigma=\frac{1}{2}r$. According to \cite{shen1995new,jing1993qualitative, Nordbo2007}, we have a generic Hopf-bifurcation at $\sigma=\frac{1}{2}r$. Small amplitude oscillations are generated at this point with a frequency $\sqrt{a_{3}/a_{1}}$. Notice that the existence of a Hopf-bifurcation in the modelling framework (\ref{logisticdistributeddelay}) is pointed out in \cite{cushing2013}. \\
\newline
Next, let us consider  the dynamical system (\ref{Gsystem1})-(\ref{Componentfunctionsvectorfieldlogistic}). This system possesses two equilibrium points, the trivial state $P_{\ast}=[0,0,0,0,0,0,0]$ and a nontrivial equilibrium point which is denoted by $P_{e}=[x_{e},V_{1,e},V_{2,e}, u_{1,e}, v_{1,e}, u_{2,e}, u_{2,e}]$.
We readily find that the coordinates of $P_{e}$ are given as
\begin{subequations}
\label{equilibriumlogistic}
\begin{eqnarray}
&V_{1,e}=V_{2,e}=x_{e}=\frac{(1+\theta)^{2}}{(1+\theta)^{2}-\varepsilon(\theta-1)}K&\label{equilibriumlogistica}\\
&&\nonumber\\
&u_{1,e}=\frac{1+\theta}{(1+\theta)^{2}-\varepsilon(\theta-1)}K
,\quad v_{1,e}=\frac{\theta^{1/2}(1+\theta)}{(1+\theta)^{2}-\varepsilon(\theta-1)}K&\label{equilibriumlogisticb}\\
&&\nonumber\\
&u_{2,e}=\frac{1-\theta}{(1+\theta)^{2}-\varepsilon(\theta-1)}K,\quad v_{2,e}=\frac{2\theta^{1/2}}{(1+\theta)^{2}-\varepsilon(\theta-1)}K&\label{equilibriumlogisticc}
\end{eqnarray}
\end{subequations}
where the dimensionless positive parameter $\theta$ is defined as the square of the ratio between the delay time scale $T_{\sigma}\equiv 1/\sigma$ and the period $T_{\Omega}\equiv 1/\Omega$ of the oscillations in the time history, i.e.,
\begin{equation}
\theta\equiv \big(T_{\sigma}/T_{\Omega}\big)^{2}=\big(\Omega/\sigma\big)^{2}\label{equilibriumlogisticd}
\end{equation}
Thus the presence of a finite amplitude oscillation in the time history causes an $\varepsilon$-dependent change in the equilibrium state as compared with the equilibrium state in the standard case without oscillations in the time history. The case $\epsilon=0$ is reobtained through the decoupling of $x$ and $u_2$, not from a direct reduction in dimensionality.\\
\newline
We next study the stability of the equilibria, starting with the trivial one $P_{\ast}$. The Jacobian of the vector field $\mathbf{F}$ evaluated at this point is given as
 \begin{equation}
D_{\textbf{x}}\textbf{F}(P_{\ast})
=\left[
  \begin{array}{cccccccc}
r& 0&0& 0& 0&0&0 \\
\sigma& -\sigma&0&0&0&0&0\\
0&\sigma&-\sigma&0&0&0&0\\
\sigma&0&0&-\sigma&-\Omega&0&0\\
0&0&0&\Omega&-\sigma&0&0\\
0&0&0&\sigma&0&-\sigma&-\Omega\\
0&0&0&0&\sigma&\Omega&-\sigma
    \end{array}
\right]\nonumber
\end{equation}
Noticing the block structure of this matrix (with the $3\times 3$ matrix given as (\ref{matrixP0}) as upper left block), we easily derive the characteristic polynomial $\mathcal{P}_{0}$ of this matrix:
\begin{equation}
\mathcal{P}_{0}(z)=\mathcal{P}_{0}(z)\cdot\big[\mathcal{P}_{2}(z)\big]^{2}\nonumber
\end{equation}
Here $\mathcal{P}_{0}$ is given by (\ref{P0}) and
\begin{equation}
\mathcal{P}_{2}(z)=(z+\sigma)^{2}+\Omega^{2}\label{P2}
\end{equation}
is the characteristic polynomial of the $2\times 2$- matrix
\begin{equation}
\mathbb{J}_{22}=\left[
  \begin{array}{cccccccc}
-\sigma&-\Omega\\
\Omega&-\sigma
    \end{array}
\right]\nonumber
\end{equation}
Hence we have that $z_{\pm}=-\sigma\pm i\Omega$ are zeros of order $2$ of the characteristic polynomial $\mathcal{P}_{7}$ and that $Re[z_{\pm}]=-\sigma<0$. Since the eigenvalue $z=r>0$ we conclude that the equilibrium point $P_{\ast}$ is unstable for all $\varepsilon$.\\
\newline
Next we consider the stability problem for the equilibrium state $P_{e}$.
% Introduce the vector field $\mathbf{F}:\mathbb{R}^{7}\rightarrow\mathbb{R}^{7}$ defined by
%\begin{eqnarray}
%&\mathbf{F}(\mathbf{X})=(F_{1}(\mathbf{X}),F_{2}(\mathbf{X}),F_{3}(\mathbf{X}),F_{4}(\mathbf{X}),F_{5}(\mathbf{X}),F_{6}(\mathbf{X}),F_{7}(\mathbf{X})),&\nonumber\\
%&&\label{componentslogistic}\\
%%\end{eqnarray}
%where the component functions are given as
%\begin{subequations}
%\label{componentfunctionsvectorfieldlogistic}
%\begin{eqnarray}
%&F_{1}(\mathbf{X})=r N(1-\frac{V_{2}+\varepsilon u_{2}}{\bar K})&\label{componentfunctionsvectorfieldlogistica}\\
%&&\nonumber\\
%&F_{2}(\mathbf{X})=-\sigma V_{1}+\sigma N,\quad F_{3}(\mathbf{X})=-\sigma V_{2}+\sigma V_{1}&\label{componentfunctionsvectorfieldlogisticb}\\
%&&\nonumber\\
%&F_{4}(\mathbf{X})=-\sigma u_{1}-\Omega v_{1}+\sigma N,\quad F_{5}(\mathbf{X})=-\sigma v_{1}+\Omega u_{1}&\label{componentfunctionsvectorfieldlogisticc}\\
%&&\nonumber\\
%&F_{6}(\mathbf{X})=-\sigma u_{2}-\Omega v_{2}+\sigma u_{1},\quad F_{7}(\mathbf{X})=-\sigma v_{2}+\Omega u_{2}+\sigma v_{1}&\label{componentfunctionsvectorfieldlogisticd}
%\end{eqnarray}
%\end{subequations}
We first compute the Jacobian $D_{\mathbf{X}}\mathbf{F}$ of the vector field $\mathbf{F}$ evaluated at $P_{e}$ and obtain
\begin{equation}
D_{\textbf{x}}\textbf{F}(P_{e})
=\left[
  \begin{array}{cccccccc}
0& 0& -U_{e}& 0& 0&-\varepsilon U_{e}&0 \\
\sigma& -\sigma&0&0&0&0&0\\
0&\sigma&-\sigma&0&0&0&0\\
\sigma&0&0&-\sigma&-\Omega&0&0\\
0&0&0&\Omega&-\sigma&0&0\\
0&0&0&\sigma&0&-\sigma&-\Omega\\
0&0&0&0&\sigma&\Omega&-\sigma
    \end{array}
\right], \quad U_{e}\equiv\frac{r x_{e}}{K}\nonumber
\end{equation}
The characteristic polynomial $\mathcal{P}_{7}$ of $D_{\textbf{x}}\textbf{F}(P_{e})$ is a seventh-order polynomial
\begin{equation}
\mathcal{P}_{7}(z)=z^{7}+a_{1}z^{6}+a_{2}z^{5}+a_{3}z^{4}+a_{4}z^{3}+a_{5}z^{2}+a_{6}z+a_{7}\label{characteristicpolynomiallogistic}
\end{equation}
where the coefficients $a_{j}$, ($j=1,2,\cdots,7$) are cumbersome expressions of the entries in $D_{\textbf{x}}\textbf{F}(P_{e})$. See \ref{Appendix B}.\\
\newline
% We notice that the equilibrium point is unstable if $a_{7}<0(\Leftrightarrow \vert D_{\textbf{x}}\textbf{F}(P_{e})\vert>0)$. 
Simple computation yields
\begin{equation}
a_{7}=-\vert D_{\textbf{x}}\textbf{F}(P_{e})\vert=-U_{e}\sigma^{6}\big[\varepsilon(\theta-1)-(\theta+1)^{2}\big]\nonumber
\end{equation}
where $\theta$ is defined by means of (\ref{equilibriumlogisticd}). We first notice that the coordinates (\ref{equilibriumlogistica})-(\ref{equilibriumlogisticc}) of the equilibrium point $P_{e}$ approach infinity when $a_{7}\rightarrow 0$. In the further discussion we proceed as follows: Introduce the function $\Psi:[0,1)\cup(1,+\infty)\rightarrow\mathbf{R}$ defined by 
\begin{equation}
\Psi[\theta]\equiv\frac{(\theta+1\big)^{2}}{\theta-1}\nonumber
\end{equation}
The points in the $\theta,\varepsilon$-plane satisfying the condition $\vert D_{\textbf{x}}\textbf{F}(P_{e})\vert<0$ are given by $0\leq\theta<1,\varepsilon>\Psi[\theta]$, the vertical line $\theta=1$ and the region $\theta>1,\varepsilon<\Psi[\theta]$. The complementary subset in the $\theta,\varepsilon$-plane obeys the condition $\vert D_{\textbf{x}}\textbf{F}(P_{e})\vert>0$. According to the definition (\ref{equilibriumlogisticd}), we must have $\theta\geq 0$.
%&=U_{e}\sigma^{6}\big((\frac{\Omega}{\sigma})^{2}-1\big)\big(\varepsilon-\varepsilon_{s}\big),\quad \varepsilon_{s}\equiv\frac{\big((\frac{\Omega}{\sigma})^{2}+1\big)^{2}}{(\frac{\Omega}{\sigma})^{2}-1}&\nonumber
The latter regime corresponds to unstable equilibrium states, since in this case $a_{7}<0$ which means that there is at least one positive zero of the characteristic polynomial $\mathcal{P}_{7}$.\\
\newline
% saddle node bifurcations take place when  $\vert D_{\textbf{x}}\textbf{F}(P_{e})\vert=0$. Noticing that $\varepsilon=0$ yields $\vert D_{\textbf{x}}\textbf{F}(P_{e})\vert<0$, we conclude that we have no saddle node bifurcation in the case of no oscillations in the memory kernel. The same holds true for $\theta=1 (\Leftrightarrow\Omega=\sigma)$. For $\theta\neq1$ we get $\vert D_{\textbf{x}}\textbf{F}(P_{e})\vert=0$ if and only if $\varepsilon=\Phi[\theta]$. We conclude that this curve may give rise to saddle-node bifurcations. For a given $\varepsilon$ the number of solutions to the equation $\varepsilon=\Phi[\theta]$ varies from $0$ to $2$: When $\varepsilon>8$, the number of solutions is $2$, i.e. we have $1<\theta_{1}<3<\theta_{2}$ for which $\Phi[\theta_{1}]=\Phi[\theta_{2}]=\varepsilon$. For $\varepsilon< -1$ we get a unique $\theta\in(0,1)$ such that $\Phi[\theta]=\varepsilon$. For $\varepsilon=8$ we get the unique solution $\theta=3$. For $-1<\varepsilon<8$, we have no solutions. This means that when the magnitude of the amplitude $\varepsilon$ of the finite oscillations in the memory function $\alpha$ exceeds a threshold (which depends on $\Omega/\sigma (\neq 1)$) we can get saddle-node bifurcations in our system. The saddle node bifurcation diagram can be constructed by exploiting the MATCONT - capabilities in MATLAB. \\
%\newline
Assume now that $\vert D_{\textbf{x}}\textbf{F}(P_{e})\vert<0$. For the purpose of carrying out the stability analysis of the nontrivial equilibrium point $P_{e}$ in this case we first collect the  Routh-Hurwitz matrices $\mathbf{D}_{j}$, ($j=1,2,\cdots,7$) of $\mathcal{P}_{7}$, i.e., 
\begin{eqnarray}
&\mathbf{D}_{1}=[a_{1}],\quad \mathbf{D}_{2}=\left[
  \begin{array}{cccccccc}
a_{1}& a_{3}\\
1& a_{2}
    \end{array}
\right],\quad \mathbf{D}_{3}=\left[
  \begin{array}{cccccccc}
a_{1}& a_{3}&a_{5}\\
1& a_{2}& a_{4}\\
0& a_{1}&a_{3}
    \end{array}
\right]&\nonumber\\
& \mathbf{D}_{4}=\left[
  \begin{array}{cccccccc}
a_{1}& a_{3}&a_{5}&a_{7}\\
1& a_{2}& a_{4}& a_{6}\\
0& a_{1}&a_{3}& a_{5}\\
0& 1& a_{2} & a_{4}
    \end{array}
\right],\quad  \mathbf{D}_{5}=\left[
  \begin{array}{cccccccc}
a_{1}& a_{3}&a_{5}&a_{7}&0\\
1& a_{2}& a_{4}&a_{6}&0\\
0& a_{1}&a_{3}& a_{5}&a_{7}\\
0& 1& a_{2} & a_{4}&a_{6}\\
0& 0&a_{1}&a_{3}&a_{5}
    \end{array}
\right]&\nonumber\\
&\mathbf{D}_{6}=\left[
  \begin{array}{cccccccc}
a_{1}& a_{3}&a_{5}&a_{7}&0&0\\
1& a_{2}& a_{4}&a_{6}&0&0\\
0& a_{1}&a_{3}& a_{5}&a_{7}&0\\
0& 1& a_{2} & a_{4}&a_{6}&0\\
0& 0&a_{1}&a_{3}&a_{5}& a_{7}\\
0 & 0&1&a_{2}&a_{4}&a_{6}  \end{array}
\right],\quad 
\mathbf{D}_{7}=\left[
  \begin{array}{cccccccc}
a_{1}& a_{3}&a_{5}&a_{7}&0&0&0\\
1& a_{2}& a_{4}&a_{6}&0&0&0\\
0& a_{1}&a_{3}& a_{5}&a_{7}&0&0\\
0& 1& a_{2} & a_{4}&a_{6}&0&0\\
0& 0&a_{1}&a_{3}&a_{5}& a_{7}&0\\
0 & 0&1&a_{2}&a_{4}&a_{6} &0\\
0&0&0&a_{1}&a_{3}&a_{5}& a_{7}
 \end{array}
\right]&\nonumber
\end{eqnarray}
and then compute the corresponding Routh-Hurwitz determinants $\vert\mathbf{D}_{j}\vert$, ($j=1,2,\cdots,7$). From the Routh-Hurwitz criterion it follows that all the zeros of $\mathcal{P}_{7}$ are located in the left $z$-plane if and only if $\vert\mathbf{D}_{j}\vert>0$ for all $j=1,2,\cdots,7$. See \cite{hurwitz1964conditions}. This means that equilibrium point $P_{e}$ is asymptotically stable if and only if $\vert\mathbf{D}_{i}\vert>0$ for all $j=1,2,\cdots,7$.\\
\newline
We first examine the situation when $\varepsilon=0$.  The Jacobian matrix of the vector field $\mathbf{F}$ evaluated at the equilibrium point $P_{e}(\varepsilon=0)$ becomes the block lower triangular matrix
\begin{equation}
D_{\textbf{x}}\textbf{F}_{0}(P_{e}(\varepsilon=0))
=\left[
  \begin{array}{cccccccc}
0& 0& -r& 0& 0&0&0 \\
\sigma& -\sigma&0&0&0&0&0\\
0&\sigma&-\sigma&0&0&0&0\\
\sigma&0&0&-\sigma&-\Omega&0&0\\
0&0&0&\Omega&-\sigma&0&0\\
0&0&0&\sigma&0&-\sigma&-\Omega\\
0&0&0&0&\sigma&\Omega&-\sigma
    \end{array}
\right]\label{Jacobian7D0}
\end{equation}
in this case. We explore the spectral properties of this matrix in order to assess the linear stability of the equilibrium point $P_{e}(\varepsilon=0)$.  A notable feature of the matrix (\ref{Jacobian7D0})  is its block structure containing the $3\times 3$ upper left block matrix $\mathbb{J}_{33}$ defined by (\ref{Jacobian3D})
and the $4\times 4$ lower right block matrix
\begin{equation}
\mathbb{J}_{44}=\left[
  \begin{array}{cccccccc}
-\sigma&-\Omega&0&0\\
\Omega&-\sigma&0&0\\
\sigma&0&-\sigma&-\Omega\\
0&\sigma&\Omega&-\sigma
    \end{array}
\right]\nonumber
\end{equation}
The block structure of the Jacobian matrix $D_{\textbf{x}}\textbf{F}_{0}(P_{e}(\varepsilon=0))$ implies that the characteristic polynomial $\mathcal{P}_{7}$ of  (\ref{Jacobian7D0}) can be factorized as
\begin{equation}
\mathcal{P}_{7}(z)=\mathcal{P}_{3}(z)\cdot[\mathcal{P}_{2}(z)]^{2}\nonumber
\end{equation}
where $\mathcal{P}_{3}$ is given as the characteristic polynomial (\ref{characteristicpol3}) and $\mathcal{P}_{2}$ is defined by (\ref{P2}).
Since $z=z_{\pm}\equiv-\sigma\pm i\Omega$ are zeros of order $2$ of the characteristic polynomial $\mathcal{P}_{7}$ and that $Re[z_{\pm}]=-\sigma<0$, we conclude that the linear stability of the equilibrium $P_{e}(\varepsilon=0)$ is as expected  determined by the spectral properties of the Jacobian $\mathbb{J}_{33}$: For $\sigma>\frac{1}{2}r$ ($0<\sigma<\frac{1}{2}r$) the equilibrium point will be asymptotically stable (unstable).\\
\newline 
In order to investigate the impact of the oscillatory time history on the Hopf bifurcation problem characterized by $\sigma=\frac{1}{2}r$, we view  the delay parameter $\sigma$ as a control parameter. The actual impact is stored in the amplitude parameter $\varepsilon$. Following the ideas in \cite{shen1995new,jing1993qualitative, Nordbo2007}, the existence of a generic Hopf-bifurcation concerns the study of the equation
\begin{equation}
\vert\mathbf{D}_{6}\vert(\sigma,\varepsilon)=0\nonumber
\end{equation}
For $\varepsilon=0$, this equation will have the unique solution $\sigma=\frac{1}{2}r$. Moreover, the non-transversality condition 
$\partial_{\sigma}\big[\vert\mathbf{D}_{6}\vert\big](\frac{1}{2}r,0)\neq 0$ is fulfilled. See \ref{Appendix B}. The Routh-Hurwitz determinant $\vert\mathbf{D}_{6}\vert$ is a smooth function in an open neighborhood of $(\frac{1}{2}r,0)$. Hence, by appealing to the implicit function theorem there is a unique $C^{1}$-function $\phi:I_{0}\rightarrow\mathbb{R}_{+}$ where $I_{0}$ is an open $\varepsilon$-interval about $0$ for which we have
\begin{subequations}
\label{GenHopf}
\begin{eqnarray}
&\vert\mathbf{D}_{6}\vert(\sigma,\varepsilon)=0, \quad\varepsilon\in I_{0}&\label{GenHopfa}\\
&&\nonumber\\
&\sigma=\phi(\varepsilon)=\frac{1}{2}r-\bigg(\partial_{\sigma}\big[\vert\mathbf{D}_{6}\vert\big](\frac{1}{2}r,0)\bigg)^{-1}\partial_{\varepsilon}\big[\vert\mathbf{D}_{6}\vert\big]
(\frac{1}{2}r,0)\cdot\varepsilon+\mathcal{O}(\varepsilon^{2})&\label{GenHopfb}
\end{eqnarray}
\end{subequations}
In accordance with \ref{Appendix B} all the Routh-Hurwitz determinants $\vert\mathbf{D}_{j}\vert$ ($j=1,2,\cdots,5$) are smooth and strictly positive functions of $(\sigma,\varepsilon)$ in an open neighborhood of $(\frac{1}{2}r,0)$. This result together with (\ref{GenHopf}) enable us to conclude that the generic Hopf-bifurcation located at $\sigma=\frac{1}{2}r$ for $\varepsilon=0$ extends to the finite $\varepsilon$-regime.\\
\newline
\begin{figure}
    \centering
    \includegraphics[width=0.8\linewidth]{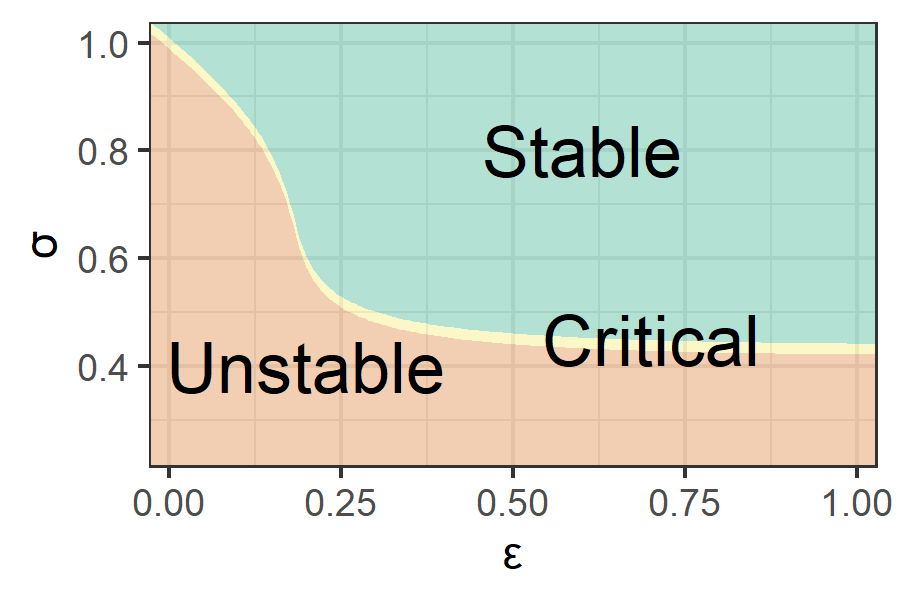}
    \caption{Phase diagram of the logistic equation with oscillations in the time history for different values of $\varepsilon$ and $\sigma$. The input parameters are $r=2$, $K=1$, and $\Omega=0.8$. The critical line satisfies (\ref{GenHopf}) as well as $|\mathbf{D}_j|>0$ for $j=1\dots 5$.}
    \label{Fig:PhaseDiagramLogistic}
\end{figure}
\begin{figure}
    \centering
    \includegraphics[width=\linewidth]{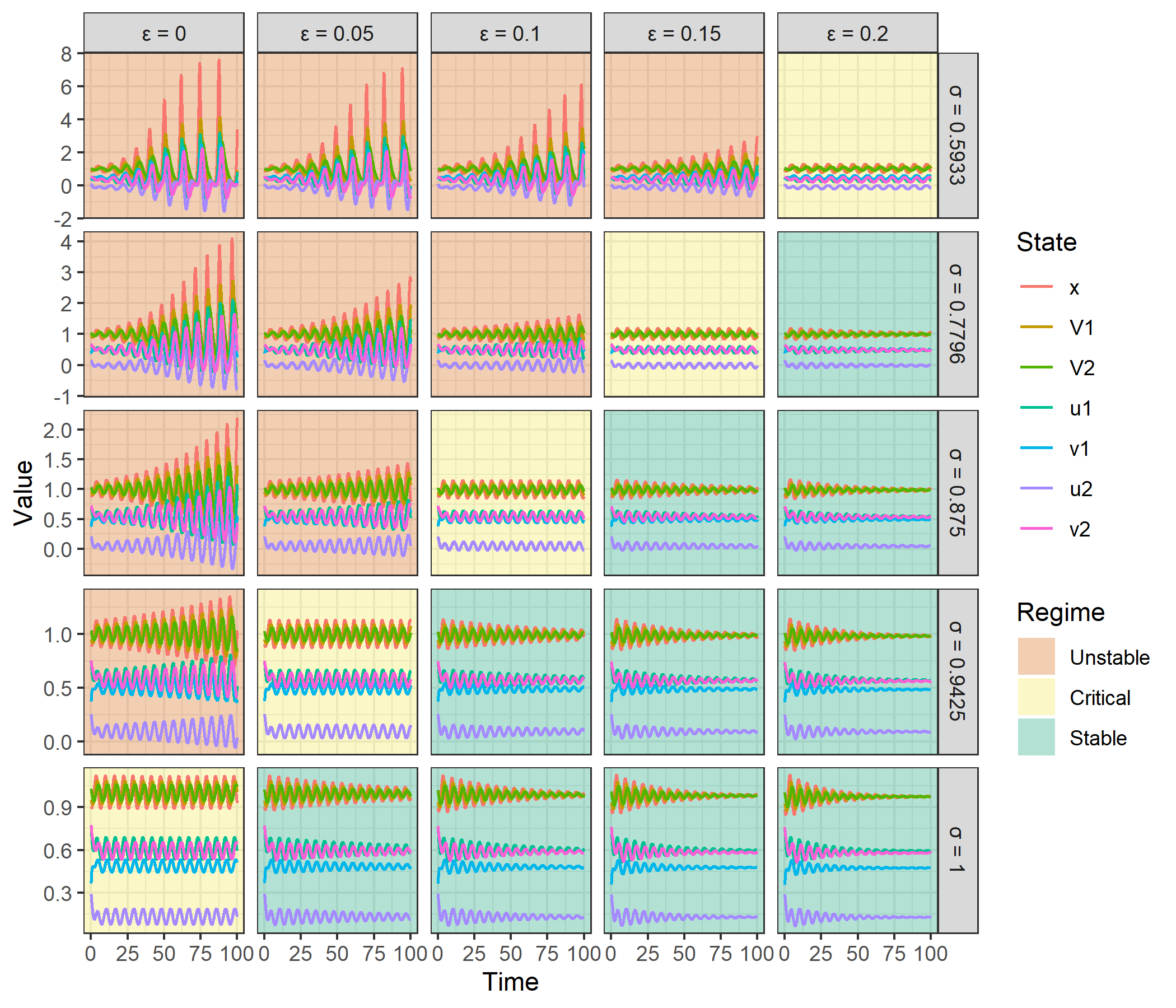}
    \caption{Numerical solution of the logistic equation with oscillations in the time history for different values of $\varepsilon$ and $\sigma$. That is, the system of equations (\ref{rate1a}), (\ref{rate2}) and (\ref{logisticdelay}). The input parameters are $r=2$, $K=1$, and $\Omega=0.8$. The combinations of $\varepsilon$ and $\sigma$ are chosen along the critical line in Figure \ref{Fig:PhaseDiagramLogistic}, and the background color is based on that figure as well. Note how the oscillations of all state variables increase in magnitude in the unstable regime, but decrease in the stable one.}
    \label{Fig:NumericalLogistic}
\end{figure}
 The oscillations generated at the Hopf-bifurcation point for $\varepsilon=0$-case will continuously deform to oscillations in the perturbed case $0<\varepsilon\ll 1$. Here we will stress the importance of studying numerically the oscillatory effect in the memory kernel on the Hopf-bifurcation: In the phase diagram depicted in Fig.\ref{Fig:PhaseDiagramLogistic} the points in the green colored region satisfies $\vert\mathbf{D}_{6}\vert(\sigma,\varepsilon)>0$, which means that the corresponding equilibrium states are asymptotically stable, whereas points in the complementary red colored region satisfy $\vert\mathbf{D}_{6}\vert(\sigma,\varepsilon)<0$. The latter regime produces unstable equilibrium states. The yellow colored separatrix curve which is given by $\vert\mathbf{D}_{6}\vert(\sigma,\varepsilon)=0$ yields points in the $\sigma,\varepsilon$-plane representing Hopf-bifurcation points. The input parameters used in the numerical simulations of the system (\ref{rate1a}), (\ref{rate2}) and (\ref{logisticdelay}) leading to Fig.\ref{Fig:NumericalLogistic} are based on Fig. \ref{Fig:PhaseDiagramLogistic}.\\ 
\newline
It is interesting to investigate the details of the stability characteristics of the equilibrium state (\ref{equilibriumlogistic}) and compare the outcome of this analysis with asymptotical results and numerical explorations, i.e., to see how far we can stretch the predictions based on the asymptotical analysis.
Thus, let us assume that $0<\varepsilon \ll 1$, which means that the effect of oscillations in the integral kernel is assumed to be weak. We decompose the vector field $\mathbf{F}$ defining the dynamical system as
\begin{equation}
\mathbf{F}(\mathbf{X},\varepsilon)=\mathbf{F}^{(0)}(\mathbf{X})+\varepsilon\mathbf{F}^{(1)}(\mathbf{X})\nonumber
\end{equation}
where
\begin{subequations}
\label{unperturbed}
\begin{eqnarray} 
&\mathbf{F}_{0}(\mathbf{X})=(F_{1}(\mathbf{X}),F_{2}(\mathbf{X}),F_{3}(\mathbf{X}),F_{4}(\mathbf{X}),F_{5}(\mathbf{X}),F_{6}(\mathbf{X}),F_{7}(\mathbf{X}))&\label{unperturbeda}\\
&&\nonumber\\
&F_{1}(\mathbf{X})=r x(1-\frac{V_{2}}{\bar K})&\label{unperturbedb}\\
&&\nonumber\\
&F_{2}(\mathbf{X})=-\sigma V_{1}+\sigma x,\quad F_{3}(\mathbf{X})=-\sigma V_{2}+\sigma V_{1}&\label{unperturbedc}\\
&&\nonumber\\
&F_{4}(\mathbf{X})=-\sigma u_{1}-\Omega v_{1}+\sigma x,\quad F_{5}(\mathbf{X})=-\sigma v_{1}+\Omega u_{1}&\label{unperturbedd}\\
&&\nonumber\\
&F_{6}(\mathbf{X})=-\sigma u_{2}-\Omega v_{2}+\sigma u_{1},\quad F_{7}(\mathbf{X})=-\sigma v_{2}+\Omega u_{2}+\sigma v_{1}&\label{unperturbede}
\end{eqnarray}
\end{subequations}
and 
\begin{equation}
\mathbf{F}^{(1)}(\mathbf{X})=(-r\frac{xu_{2}}{\bar K},0,0,0,0,0,0)\label{perturbedcomp}
\end{equation}
The dynamical system
\begin{equation}
\dot{\mathbf{X}}=\mathbf{F}^{(0)}(\mathbf{X})+\varepsilon\mathbf{F}^{(1)}(\mathbf{X}),\quad 0<\varepsilon\ll 1\label{perturbeddynsyst}
\end{equation}
can be dealt with by using regular perturbation theory, see \cite{Vasileva1995}. The asymptotic approximation of the solution to this system is given as
\begin{equation}
\mathbf{X}(t)=\mathbf{X}_{0}(t)+\varepsilon\mathbf{X}_{1}(t)+\mathcal{O}(\varepsilon^{2})\nonumber
\end{equation}
where the terms satisfy the hierarchy of ODEs
\begin{subequations}
\label{1}
\begin{eqnarray}
&\dot{\mathbf{X}}_{0}=\mathbf{F}^{(0)}(\mathbf{X}_{0})&\label{1a}\\
&&\nonumber\\
&\dot{\mathbf{X}}_{1}=D_{\mathbf{X}}\mathbf{F}^{(0)}(\mathbf{X}_{0})\cdot\mathbf{X}_{1}+\mathbf{F}^{(1)}(\mathbf{X}_{0})&\label{1b}\\
&\vdots&\nonumber
\end{eqnarray}
\end{subequations}
Here $D_{\mathbf{X}}\mathbf{F}^{(0)}(\mathbf{X}_{0})$ denotes the Jacobian of the vector field $\mathbf{F}^{(0)}$ evaluated at the leading order solution $\mathbf{X}_{0}$.
% Let us consider the leading order system in this hierarchy, i.e., (\ref{1a}).\\ 
%One notices that this system contains 3D subsystem for the component functions $N$, $V_{1}$ and $V_{2}$ that does not depend on the remaining four auxilliary variables $u_{1}$, $v_{1}$, $u_{2}$ and $v_{2}$. The internal equilibrium state $(\bar K,\bar K,\bar K)$ of this subsystem is subject to a generic Hopf - bifurcation when $\sigma=\frac{1}{2}r$. The same state is asymptotically stable within the framework of this subsystem when $\sigma>\frac{1}{2}r$ and unstable in the complementary regime $0<\sigma<\frac{1}{2}r$. The detection procedure for the Hopf - bifurcation is outlined in \cite{shen1995new,jing1993qualitative, Nordbo2007}. It is based on violation of the Routh - Hurwitz criterion. It is of particular interest to explore the impact of the oscillatory effects on the Hopf - bifurcation state $\sigma=\frac{1}{2}r$.\\
%\newline
%For the purpose of investigating the impact of oscillatory effects we determine $N, V_{1},V_{2}$ first and make use of these solutions as inputs in the leading order rate equations for $u_{1}$, $v_{1}$, $u_{2}$ and $v_{2}$. The perturbation $\mathbf{X}_{1}$ is then obtained from the next equation in the hierarchy. \\
%\newline
The unperturbed equilibrium state  $P_{e}(\varepsilon=0)$ is a hyperbolic equilibrium in the parameter regime $\sigma\neq\frac{1}{2}r$. In that case the perturbed system (\ref{perturbeddynsyst}) possesses an equilibrium point $P_{e}(\varepsilon\neq 0)$ which appears as a smooth $\varepsilon$-deformation of the equilibrium point $P_{e}(\varepsilon=0)$ of the unperturbed system (\ref{1a}) and which also is hyperbolic. The Hartman-Grobmans theorem implies that the phase portrait of the unperturbed system (\ref{1a}) in the vicinity of $P_{e}(\varepsilon=0)$ is mapped one-to-one and onto the phase portrait of the unperturbed  linearized system
\begin{equation}
\dot{\mathbf{Y}}_{0}=D_{\textbf{x}}\textbf{F}^{(0)}(P_{e}(\varepsilon=0))\cdot\mathbf{Y}_{0}\label{linearizedunpert}
\end{equation}
in the vicinity of $\mathbf{Y}_{0}=\mathbf{0}$. We notice that the phase portrait of (\ref{linearizedunpert}) is mapped one-to-one and onto the phase portrait of the perturbed linearized system
\begin{equation}
\dot{\mathbf{Y}}_{\varepsilon}=D_{\textbf{x}}\textbf{F}(P_{e}(\varepsilon\neq 0))\cdot\mathbf{Y}_{\varepsilon},\quad 0<\varepsilon\ll 1\label{linearizedpert}
\end{equation}
in the vicinity of $\mathbf{Y}_{\varepsilon}=\mathbf{0}$. Hartman-Grobmans theorem then implies that the latter phase portrait is mapped one-to-one and onto the phase portrait of the perturbed system (\ref{perturbeddynsyst}) in the vicinity of $P_{e}(\varepsilon\neq 0)$.
Thus the phase portrait of the perturbed system  (\ref{perturbeddynsyst}) in the vicinity of the perturbed equilibrium point appears as a smooth $\varepsilon$-deformation of the corresponding unperturbed system (\ref{1a}) in the vicinity of $P_{e}(\varepsilon=0)$. \\
\newline
To sum up: We make use of the  Hartman-Grobmans theorem to conclude that the local dynamics about $P_{e}(\varepsilon=0)$ in the hyperbolic case $\sigma\neq\frac{1}{2}r$ will continuously deform to the local dynamics about $P_{e}(\varepsilon\neq 0)$ ($0<\varepsilon\ll 1$).  We have also shown that the Hopf-bifurcation point $\sigma=\frac{1}{2}r,\varepsilon=0$ will continuously deform to a Hopf-bifurcation in the perturbed case  $0<\varepsilon\ll 1$ in accordance with the violation of the Routh-Hurwitz criterion embodied in (\ref{GenHopf}). Both these results are consistent with the fact that the mapping from the Banach space of absolute integrable memory functions to the Banach space of absolute continuous solutions of the distributed delay system (\ref{logisticdistributeddelay}) is continuous with respect to the metrics (Theorem \ref{stabilitytheorem} in \ref{Continuous dependence on the time history}).

\section{Linear Chain Trick for systems with oscillating dependence on the time history}\label{LCT}

Motivated by the examples elaborated in the previous section, we will now work out the LCT-algorithm for a general dynamical system with oscillations in the time history. Here we follow and generalize the set-up in \cite{ponosov2004thew}. We consider a $D$-dimensional real-valued distributed delay system on the form
\begin{equation}\label{system1}
\left.\begin{aligned} 
&\dot{\textbf{x}}(t)=\textbf{F}(\textbf{x}(t),(\Re{\textbf{x}})(t)),\quad t\in\mathcal{I}_{+}\equiv (0,t^{\ast})&\\
&\textbf{x}(\tau)=\textbf{u}(\tau),\tau\in\mathcal{I}_{-}\equiv (-\infty,0]&
\end{aligned} \right\}
\end{equation}
where $\textbf{u}:(-\infty,0]\rightarrow\mathbb{R}^{D}$ is a known vector valued initial function, $\textbf{x}(t)=[x_{1}(t), x_{2}(t),\cdots, x_{D}(t)]^{T}\in \mathbb{R}^{D}$ is the solution of the system, and $0<t^{\ast}\leq\infty$. The delay operator $\Re$ is defined as
\begin{equation}
\label{system2}
\left.\begin{aligned}
&(\Re{\textbf{x}})(t)=\sum_{k=1}^{N}c_{k}\textbf{z}_{k}(t),\quad \textbf{z}_{k}(t)=\int\limits_{-\infty}^{t}\alpha_{k}(t-s)\textbf{x}_{d}(s)ds\in\mathbb{R}^{D}&\\
&\textbf{x}_{d}(t)\equiv\mathbb{J}\cdot\textbf{x}(t)=[x_{1}(t), x_{2}(t),\cdots, x_{d}(t),0,0,\cdots,0]^{T}&\\
&\alpha_{k}(t)=\frac{\sigma^{k}}{(k-1)!}t^{k-1}e^{-\sigma t}&\\
&\hspace{80pt}\times\big[a_{k}+\sum_{n=1}^{M_{k}}(\varepsilon_{n}^{(k)}\cos(\omega_{n}^{(k)}t)+\mu_{n}^{(k)}\sin(\omega_{n}^{(k)}t))\big]&
\end{aligned} \right\}
\end{equation}
where $a_{k}$, $\varepsilon_{n}^{(k)}$, $\mu_{n}^{(k)}$ and $\omega_{n}^{(k)}$ define sequences of real numbers and $\sigma>0$. Here the expansion coefficients $c_{k}$ (k=1,2,...,N) are assumed to be real-valued. $\mathbb{J}$ is the $D\times D$-projection matrix defined by
\begin{equation}\label{matrix}
\mathbb{J}\equiv
\begin{bmatrix}
\mathbb{I}_{d}&\vdots &\mathbb{O}_{d\times (D-d)}\\
\cdots\cdots&\cdots\vdots\cdots&\cdots\cdots\\
\mathbb{O}_{(D-d)\times d}&\vdots&\mathbb{O}_{(D-d)\times (D-d)}
\end{bmatrix},\quad 1\leq d\leq D
\end{equation}
where $\mathbb{I}_{d}$ is the unit $d\times d$-matrix and $\mathbb{O}_{p\times q}$ the $p\times q$-zero matrix. This means that we have accounted for the assumption that the distributed delay effects in (\ref{system1})-(\ref{matrix}) are a function of the $d$ first coordinates of $\textbf{x}$. This means that the memory kernel in each term of the linear combination defining the delay term $\Re{\textbf{x}}$ is a perturbed Erlang distribution where the perturbation is a finite superposition of finite amplitude temporal oscillations.\\
\newline
We have tacitly assumed that the system (\ref{system1}) - (\ref{system2}) is locally well-posed, i.e., that the solution exists, is unique and depends continuously on the initial data for $t\in\mathcal{I}_{+}$. We will achieve this by assuming the function $\textbf{F}:U\rightarrow\mathbb{R}^{D}$ defining the right hand side of our system to be a smooth function. Here $U$ is an open, connected, convex Euclidean subset of the product space $\mathbb{R}^{D}\times\mathbb{R}^{D}$ for which we have $(\textbf{u}(0),(\Re{\textbf{u}})(0)),(\textbf{x}(t),(\Re{\textbf{x}})(t))\in U$ for $t\in\mathcal{I}_{+}$. See \cite{cushing2013} for details.\\
\newline
We notice that the sequence of integral kernels $\{\alpha_{k}\}_{k=1}^{N}$ can be written as
\begin{eqnarray}
&\hspace{-100pt}\alpha_{k}(t)=a_{k}\theta_{k}(t)&\nonumber\\
&+\frac{1}{2}\sum\limits_{n=1}^{M_{k}}\varepsilon_{n}^{(k)}[\varphi_{n}^{(k)}(t)+(\varphi_{n}^{(k)})^{\ast}(t)]&\label{system3}\\
&+\frac{1}{2i}\sum\limits_{n=1}^{M_{k}}\mu_{n}^{(k)}[\varphi_{n}^{(k)}(t)-(\varphi_{n}^{(k)})^{\ast}(t)]&\nonumber
\end{eqnarray}
where the sequences $\{\theta_{k}\}_{k=1}^{N}$ and  $\{\varphi_{n}^{(k)}\}_{n=1}^{M_{k}}$  are defined as
\begin{equation}
\label{system4}
\left.\begin{aligned}
&\theta_{k}(t)=\frac{\sigma^{k}}{(k-1)!}t^{k-1}e^{-\sigma t}&\\
&\varphi_{n}^{(k)}(t)=\frac{\sigma^{k}}{(k-1)!}t^{k-1}e^{\lambda_{n}^{(k)} t},\quad \lambda_{n}^{(k)}\equiv-\sigma+i\omega_{n}^{(k)}, \quad (i^{2}=-1)&
\end{aligned} \right\}
\end{equation}
and the operation $\ast$ means complex conjugation. We readily find the decomposition
\begin{equation}
\label{system5}
\left.\begin{aligned}
&\textbf{z}_{k}(t)=a_{k}\textbf{V}_{k}(t)+\frac{1}{2}\sum\limits_{n=1}^{M_{k}}\varepsilon_{n}^{(k)}[\textbf{W}_{n}^{(k)}(t)+(\textbf{W}_{n}^{(k)})^{\ast}(t)]&\\
&+\frac{1}{2i}\sum\limits_{n=1}^{M_{k}}\mu_{n}^{(k)}[\textbf{W}_{n}^{(k)}(t)-(\textbf{W}_{n}^{(k)})^{\ast}(t)]\in\mathbb{R}^{D}&\\
&\textbf{V}_{k}(t)\equiv\int\limits_{-\infty}^{t}\theta_{k}(t-s)\textbf{x}_{d}(s)ds\in\mathbb{R}^{D}&\\
&\textbf{W}_{n}^{(k)}(t)\equiv\int\limits_{-\infty}^{t}\varphi_{n}^{(k)}(t-s)\textbf{x}_{d}(s)ds\in\mathbb{R}^{D}&
\end{aligned} \right\}
\end{equation}
by combining (\ref{system2}), (\ref{system3}) and (\ref{system4}). We are now in position to make use of the linear chain trick developed in \cite{ponosov2004thew}: For the sequence $\{\textbf{V}_{k}\}_{k=1}^{N}$, we have the system of ODEs given as
\begin{equation}
\label{system6}
\left.\begin{aligned}
&\dot{\textbf{V}}_{1}(t)=-\sigma \textbf{V}_{1}(t)+\sigma \textbf{x}_{d}(t)&\\
&\dot{\textbf{V}}_{k}(t)=-\sigma \textbf{V}_{k}(t)+\sigma \textbf{V}_{k-1}(t),\quad k=2,3,\cdots,N&
\end{aligned} \right\}
\end{equation}
whereas we will produce the hierarchy
%\begin{equation}
%\label{system7}
%\left.\begin{aligned}
%&(\textbf{u}_{n}^{(1)})'(t)=-\omega_{n}^{(1)}\textbf{v}_{n}^{(1)}(t)-\sigma\textbf{u}_{n}^{(1)}(t)+\sigma \textbf{x}(t),\quad n=1,2,\cdots, M_{1}&\\
%&(\textbf{u}_{n}^{(k)})'(t)=-\omega_{n}^{(k)}\textbf{v}_{n}^{(k)}(t)-\sigma\textbf{u}_{n}^{(k)}(t)+\sigma \textbf{u}_{n-1}(t),\quad n=1,2,3,\cdots M_{k},\quad k=2,\cdots, N&\\
%&(\textbf{v}_{n}^{(1)})'(t)=\omega_{n}^{(1)}\textbf{u}_{n}^{(1}(t)-\sigma\textbf{v}_{n}^{(1)}(t),\quad n=1,2,\cdots, M_{1}&\\
%&(\textbf{v}_{n}^{(k)})'(t)=\omega_{n}^{(k)}\textbf{u}_{n}^{(k)}(t)-\sigma\textbf{v}_{n}^{(k)}(t)+\sigma \textbf{v}_{n-1}(t),\quad n=1,2,3,\cdots M_{k},\quad k=2,\cdots, N&
%\end{aligned} \right\}
%\end{equation}
\begin{equation}
\label{system7}
\left.\begin{aligned}
&\dot{\textbf{W}}_{n}^{(1)}(t)=\lambda_{n}^{(1)}\textbf{W}_{n}^{(1)}(t)+\sigma \textbf{x}_{d}(t),\quad n=1,2,3,\cdots, M_{1}&\\
&\dot{\textbf{W}}_{n}^{(k)}(t)=\lambda_{n}^{(k)} \textbf{W}_{n}^{(k)}(t)+\sigma\textbf{W}_{n}^{(k-1)}(t),&\\
& n=1,2,3,\cdots M_{k}; k=2,\cdots, N&
&\end{aligned} \right\}
\end{equation}
for the sequences of block functions $\{\{\textbf{W}_{n}^{(k)}\}_{n=1}^{M_{k}}\}_{k=1}^{N}$.\\
%Notice the structural resemblance between the hierarchies (\ref{system6}) and (\ref{system7}): By making the replacement $\dot{\textbf{V}}_{1}\rightarrow\dot{\textbf{W}}_{n}^{(1)}$ on the left hand side of (\ref{system6}) and the replacement  $-\sigma\textbf{V}_{k}\rightarrow \lambda_{n}^{(k)} \textbf{W}_{n}^{(k)}$ in the first term on the right hand side of (\ref{system6}) and the replacement $\textbf{x}_{d},\textbf{V}_{k-1}\rightarrow\textbf{x}_{d},\textbf{W}_{n}^{(k-1)}(t)$ in the second term on the right hand side of (\ref{system6}), we obtain (\ref{system7}). This is indeed to be expected due to the definitions in (\ref{system4}) of the integral kernels 
%$\theta_{k}$ and  $\varphi_{n}^{(k)}$. \\
\newline
To sum up, the distributed delay system (\ref{system1})-(\ref{system2}) can be replaced with the higher-order autonomous dynamical system
\begin{equation}
\label{system8}
\left.\begin{aligned}
&\dot{\textbf{x}}=\textbf{F}\big(\textbf{x},\sum_{k=1}^{N}c_{k}[a_{k}\textbf{V}_{k}+\sum_{n=1}^{M_{k}}(\varepsilon_{n}^{(k)}\textbf{u}_{n}^{(k)}+\mu_{n}^{(k)}\textbf{v}_{n}^{(k)})]\big)&\\
&t\in\mathcal{I}_{+}\equiv (0, t^{\ast})&\\
&\dot{\textbf{V}}_{1}=-\sigma \textbf{V}_{1}+\sigma \textbf{x}_{d}&\\
&\dot{\textbf{V}}_{k}=-\sigma \textbf{V}_{k}+\sigma \textbf{V}_{k-1}&\\ 
&k=2,3,\cdots,N&\\
&\dot{\textbf{u}}_{n}^{(1)}=-\omega_{n}^{(1)}\textbf{v}_{n}^{(1)}-\sigma\textbf{u}_{n}^{(1)}+\sigma \textbf{x}_{d}&\\
& n=1,2,\cdots, M_{1}&\\
&\dot{\textbf{u}}_{n}^{(k)}=-\omega_{n}^{(k)}\textbf{v}_{n}^{(k)}-\sigma\textbf{u}_{n}^{(k)}+\sigma \textbf{u}_{n}^{(k-1)}&\\
&n=1,2,3,\cdots M_{k},\quad k=2,3,\cdots, N&\\
&\dot{\textbf{v}}_{n}^{(1)}=\omega_{n}^{(1)}\textbf{u}_{n}^{(1}-\sigma\textbf{v}_{n}^{(1)}&\\
& n=1,2,\cdots, M_{1}&\\
&\dot{\textbf{v}}_{n}^{(k)}=\omega_{n}^{(k)}\textbf{u}_{n}^{(k)}-\sigma\textbf{v}_{n}^{(k)}+\sigma \textbf{v}_{n}^{(k-1)}&\\
& n=1,2,3,\cdots M_{k},\quad k=2,3,\cdots, N&
\end{aligned} \right\}
\end{equation}
Here we have made use of the decomposition into real and imaginary parts: $\textbf{W}_{n}^{(k)}=\textbf{u}_{n}^{(k)}+i\textbf{v}_{n}^{(k)}$, $\textbf{u}_{n}^{(k)},\textbf{v}_{n}^{(k)}\in\mathbb{R}^{D}$. \\
\newline
The initial condition of this system is given as
\begin{equation}
\label{system9}
\left.\begin{aligned}
&\textbf{x}(0)=\textbf{u}(0)&\\
&\textbf{V}_{k}(0)=\textbf{u}_{n}^{(k)}(0)=\int\limits_{-\infty}^{0}\theta_{k}(-s)\mathbb{J}\cdot\textbf{u}(s)ds&\\
& n=1,2,3,\cdots M_{k},\quad k=1,2,\cdots, N&\\
&\textbf{v}_{n}^{(k)}(0)=\int\limits_{-\infty}^{0}\text{Im}[\varphi_{n}^{(k)}(-s)]\mathbb{J}\cdot\textbf{u}(s)ds&\\
& n=1,2,3,\cdots M_{k},\quad k=1,2,\cdots, N&
\end{aligned} \right\}
\end{equation}
Here we have assumed that the initial function $\textbf{u}$ is designed in such a way that the improper integrals in (\ref{system9}) are convergent. The phase space of the system (\ref{system8}) is the Euclidean space $\mathbb{R}^{r}$ where 
\begin{equation}
r=D+d\cdot(N+2\cdot\sum\limits_{k=1}^{N}M_{k})\quad\textrm{(the dimension formula)}.\label{dimensionformula}
\end{equation}
In the process of deriving (\ref{dimensionformula}) we have taken into account that the $D-d$ last component equations in each of the auxiliary equations in (\ref{system8}) produce the trivial solutions. This is caused by the projection matrix $\mathbb{J}$ defined by means of (\ref{matrix}).
\begin{remark}
The logistic model (\ref{logisticdistributeddelay}) with distributed delay (\ref{system3logistic})-(\ref{system4logistic}) is on the form (\ref{system2})-(\ref{matrix}). We make the identifications $\textbf{x}=[x]^{T}\in\mathbb{R}$. According to (\ref{system2})-(\ref{matrix}), we get $\textbf{x}_{d}=[x]^{T}$. We have $D=1$, $d=1$, $N=2$, $c_{k}=0$ for all $k\neq 2$ and $c_{2}=a_{2}=1$. Since we have only one frequency component for each $k$, we get $M_{k}=1$ ($k=1,2$). From (\ref{dimensionformula}) we find that the dimension $r$ of the phase space of the autonomous dynamical system produced by the LCT algorithm will be $r=1+1\cdot (2+2\cdot(1+1))=7$ in this case, in agreement with the dimension of the autonomous dynamical system (\ref{rate1})-(\ref{logisticdelay}).
\end{remark}
%\begin{remark}
%The Kaldor-Kalecki  model is of the type (\ref{system1})-(\ref{matrix}) with $\textbf{x}=[y, k]^{T}$ and $\textbf{x}_{d}=[y,0]^{T}$. Hence $D=2$ and $d=1$. For memory kernels of the weak type, we have $N=1$, $a_{1}=1$, $c_{1}=\kappa_{0}$ where $\kappa_{0}$ is given by (\ref{norm1}). Since we have only one frequency component for $k=1$, we have $M_{1}=1$. The dimension of the autonomous system (\ref{system8}) is in this case $r=2+1\cdot(1+2)=5$, in accordance with our findings in Subsection \ref{Memory function of the weak type}.  Likewise, we find that $r=2+1\cdot (2+2\cdot(1+1))=8$ for the memory of the strong type by exploiting the dimension formula (\ref{dimensionformula}), which is in agreement with the findings in Subsection \ref{Memory function of the strong type}. We have $N=2$, $c_{k}=0$ for all $k\neq 2$, $a_{2}=1$ and $c_{2}=\kappa_{0}$ where $\kappa_{0}$ is given as (\ref{norm2}).
%\end{remark}

\begin{remark}
The LCT-procedure is also applicable when the coefficients $c_{1},c_{2},\cdots,c_{N}$ in (\ref{system2}) are time dependent, i.e., $c_{k}=c_{k}(t)$. The resulting system of ODEs which is on the form (\ref{system8}) will be non-autonomous in this case. 
\end{remark}

The next problem to be addressed is the existence of equilibrium states (and their respective stability) of the system (\ref{system8}). We introduce the vector field $\textbf{G}:\mathbb{R}^{r}\rightarrow\mathbb{R}^{r}$  and the state vector $\textbf{X}$ defined by (\ref{matrix})
\begin{equation}
\label{system10}
\left.\begin{aligned}
&\textbf{G}(\textbf{X})\equiv[\textbf{F}(\textbf{X}),\textbf{A}(\textbf{X}),
\textbf{B}(\textbf{X}),\textbf{C}(\textbf{X})]&\\
&&\\
&\textbf{X}=[\textbf{x},\textbf{V}_{1},\cdots,\textbf{V}_{N},\textbf{u}_{1},\cdots,\textbf{u}_{N},\textbf{v}_{1},\cdots,\textbf{v}_{N}]&\\
&&\\
&\textbf{u}_{k}\equiv[\textbf{u}_{1}^{(k)},\textbf{u}_{2}^{(k)},\cdots\textbf{u}_{M_{k}}^{(k)}],\quad n=1,2,\cdots M_{k},\quad k=1,2,3,\cdots, N&\\
&&\\
&\textbf{v}_{k}\equiv[\textbf{v}_{1}^{(k)},\textbf{v}_{2}^{(k)},\cdots\textbf{v}_{M_{k}}^{(k)}],\quad n=1,2,\cdots M_{k},\quad k=1,2,3,\cdots, N&
\end{aligned} \right\}
\end{equation}
Here
\begin{equation}
\label{system11}
\left.\begin{aligned}
&\textbf{F}(\textbf{X})\equiv\textbf{F}(\textbf{x},\sum_{k=1}^{N}c_{k}[a_{k}\textbf{V}_{k}+\sum_{n=1}^{M_{k}}(\varepsilon_{n}^{(k)}\textbf{u}_{n}^{(k)}+\mu_{n}^{(k)}\textbf{v}_{n}^{(k)})])&\\
&&\\
&\textbf{A}(\textbf{X})\equiv[\textbf{A}_{1}(\textbf{X}),\textbf{A}_{2}(\textbf{X}),\cdots,\textbf{A}_{N}(\textbf{X})]&\\
&&\\
&\textbf{B}(\textbf{X})\equiv[\textbf{B}_{1}(\textbf{X}),\textbf{B}_{2}(\textbf{X}),\cdots,\textbf{B}_{N}(\textbf{X})]&\\
&&\\
&\textbf{B}(\textbf{X})\equiv[\textbf{B}_{1}(\textbf{X}),\textbf{B}_{2}(\textbf{X}),\cdots,\textbf{B}_{N}(\textbf{X})]&
\end{aligned} \right\}
\end{equation}
The collection of $\textbf{A}$ - vectors is defined as
\begin{equation}
\label{system12}
\left.\begin{aligned}
&\textbf{A}_{1}(\textbf{X})\equiv-\sigma \textbf{V}_{1}+\sigma \textbf{x}_{d}&\\
&&\\
&\textbf{A}_{k}(\textbf{X})\equiv-\sigma \textbf{V}_{k}+\sigma \textbf{V}_{k-1},\quad k=2,3,\cdots,N&
\end{aligned} \right\}
\end{equation}
The $\textbf{B}$ - vectors are defined in the following manner:
\begin{equation}
\label{system13}
\left.\begin{aligned}
&\textbf{B}_{1}(\textbf{X})\equiv[\textbf{B}_{1}^{(1)}(\textbf{X}),\textbf{B}_{2}^{(1)}(\textbf{X}),\cdots,\textbf{B}_{M_{1}}^{(1)}(\textbf{X})]&\\
&&\\
&\textbf{B}_{2}(\textbf{X})\equiv[\textbf{B}_{1}^{(2)}(\textbf{X}),\textbf{B}_{2}^{(2)}(\textbf{X}),\cdots,\textbf{B}_{M_{2}}^{(2)}(\textbf{X})]&\\
&\vdots&\\
&\textbf{B}_{N}(\textbf{X})\equiv[\textbf{B}_{1}^{(N)}(\textbf{X}),\textbf{B}_{2}^{(N)}(\textbf{X}),\cdots,\textbf{B}_{M_{N}}^{(N)}(\textbf{X})]\\
&&\\
&\textbf{B}_{n}^{(1)}(\textbf{X})\equiv-\omega_{n}^{(1)}\textbf{v}_{n}^{(1)}-\sigma\textbf{u}_{n}^{(1)}+\sigma \textbf{x}_{d}&\\
& n=1,2,3,\cdots,M_{1}&\\
&&\\
&\textbf{B}_{n}^{(k)}(\textbf{X})\equiv-\omega_{n}^{(k)}\textbf{v}_{n}^{(k)}-\sigma\textbf{u}_{n}^{(k)}+\sigma \textbf{u}_{n}^{(k-1)}&\\
& n=1,2,3,\cdots M_{k},\quad k=2,\cdots, N&
\end{aligned} \right\}
\end{equation}
Likewise for the $\textbf{C}$ - vectors we have
\begin{equation}
\label{system14}
\left.\begin{aligned}
&\textbf{C}_{1}(\textbf{X})\equiv[\textbf{C}_{1}^{(1)}(\textbf{X}),\textbf{C}_{2}^{(1)}(\textbf{X}),\cdots,\textbf{C}_{M_{1}}^{(1)}(\textbf{X})]&\\
&&\\
&\textbf{C}_{2}(\textbf{X})\equiv[\textbf{C}_{1}^{(2)}(\textbf{X}),\textbf{C}_{2}^{(2)}(\textbf{X}),\cdots,\textbf{C}_{M_{2}}^{(2)}(\textbf{X})]&\\
&\vdots&\\
&\textbf{C}_{N}(\textbf{X})\equiv[\textbf{C}_{1}^{(N)}(\textbf{X}),\textbf{C}_{2}^{(N)}(\textbf{X}),\cdots,\textbf{C}_{M_{N}}^{(N)}(\textbf{X})]&\\
&&\\
&\textbf{C}_{n}^{(1)}(\textbf{X})\equiv\omega_{n}^{(1)}\textbf{u}_{n}^{(1)}-\sigma\textbf{v}_{n}^{(1)}&\\
& n=1,2,3,\cdots,M_{1}&\\
&&\\
&\textbf{C}_{n}^{(k)}(\textbf{X})\equiv\omega_{n}^{(k)}\textbf{u}_{n}^{(k)}-\sigma\textbf{v}_{n}^{(k)}+\sigma \textbf{v}_{n}^{(k-1)}&\\
& n=1,2,3,\cdots M_{k},\quad k=2,\cdots, N&
\end{aligned} \right\}
\end{equation}
The dynamical system $\dot{\textbf{X}}=\textbf{G}(\textbf{X})$ is the key object of our study. Let us first investigate the possibility of having at least one equilibrium point 
\begin{equation}
\textbf{X}_{e}= [\textbf{x}_{e},\textbf{V}_{1,e},\cdots,\textbf{V}_{N,e},\textbf{u}_{1,e},\cdots,\textbf{u}_{N,e},\textbf{v}_{1,e},\cdots,\textbf{v}_{N,e}]\label{equilibriumGsystem}
\end{equation}
of this system. Here the sequences $\{\textbf{u}_{k,e}\}_{k=1}^{N}$ and $\{\textbf{u}_{k,e}\}_{k=1}^{N}$ are given by means of (\ref{system10}).  If it exists, it must satisfy $\textbf{G}(\textbf{X}_{e})=\textbf{0}$. We get the following result:

\begin{thm} Let $\{\{\delta_{n}^{(k)}\}_{n=1}^{M_{k}}\}_{k=1}^{N}$ be the complex-valued sequence defined by
\begin{equation}
\delta_{n}^{(k)}\equiv\frac{(-\sigma)^{k}}{\prod\limits_{j=1}^{k}\lambda_{n}^{(j)}}=\frac{\sigma^{k}\prod\limits_{j=1}^{k}(\sigma+i\omega_{n}^{(j)})}{\prod\limits_{j=1}^{k}(\sigma^{2}+[\omega_{n}^{(j)}]^{2})},\quad \lambda_{n}^{(j)}\equiv-\sigma+i\omega_{n}^{(j)}\label{complexsequence}
\end{equation}
and assume that the component vector $\textbf{x}=\textbf{x}_{e}$ satisfies the condition
\begin{equation}
\textbf{F}(\textbf{x}_{e},\Re{\textbf{x}_{e}})=\textbf{0}
\end{equation}
where 
\begin{equation}
\Re{\textbf{x}_{e}}=\bigg(\sum_{k=1}^{N}c_{k}\big[a_{k}+\sum_{n=1}^{M_{k}}
(\varepsilon_{n}^{(k)}Re(\delta_{n}^{(k)})+\mu_{n}^{(k)}Im(\delta_{n}^{(k)}))\big]\bigg)\mathbb{J}\cdot\textbf{x}_{e}
\end{equation}
and 
\begin{equation}
\label{componentvectorsW}
\left.\begin{aligned}
&\textbf{V}_{k,e}=\textbf{x}_{d}\equiv\mathbb{J}\cdot\textbf{x}_{e},\quad k=1,2,\cdots N&\\
&&\\
&\textbf{W}_{n,e}^{(k)}\equiv\textbf{u}_{n,e}^{(k)}+i\textbf{v}_{n,e}^{(k)}\equiv\delta_{n}^{(k)}\mathbb{J}\cdot\textbf{x}_{e}&\\
& n=1,2,\cdots M_{k},\quad k=1,2,\cdots, N&
\end{aligned} \right\}
\end{equation}
Then $\textbf{X}_{e}$ defined by (\ref{equilibriumGsystem}) is an equilibrium point of the dynamical system $\dot{\textbf{X}}=\textbf{G}(\textbf{X})$.
\end{thm}
\begin{proof} Simple computation shows that $\textbf{V}_{k,e}$ satisfies the system (\ref{system6}) whereas we find that the constant component vectors $\textbf{W}_{n,e}^{(k)}$ are solutions of the system (\ref{system7}). 
\end{proof}
\begin{remark}
Interestingly, the equilibrium coordinates of the auxiliary vectors  $\textbf{u}_{1,e},\cdots,\textbf{u}_{N,e},\textbf{v}_{1,e},\cdots,\textbf{v}_{N,e}$ depend on the sequence of frequencies $\{\{\omega_{n}^{(k)}\}_{n=1}^{M_{k}}\}_{k=1}^{N}$ and the decay rate $\sigma$ that parameterise the sequence of time history kernels $\{\alpha_{k}\}_{k=1}^{N}$. We also observe that $\vert\delta_{n}^{(k)}\vert\leq 1$ ($n=1,2,\cdots M_{k};k=1,2,\cdots N$, from which it follows that the auxiliary vectors $\textbf{u}_{n,e}^{(k)}$ and $\textbf{v}_{n,e}^{(k)}$ are uniformly bounded: $\|\textbf{u}_{n,e}^{(k)}\|\leq \|\mathbb{J}\cdot\textbf{x}_{e}\|$ and $\|\textbf{v}_{n,e}^{(k)}\|\leq\|\mathbb{J}
\cdot\textbf{x}_{e}\|$. 
\end{remark}
We carry out the linear stability analysis of the equilibrium state $\textbf{X}_{e}$ in the standard way by investigating the spectral properties of the Jacobian $D_{\textbf{X}}\textbf{G}$ of the vector field $\textbf{G}$ evaluated at $\textbf{X}_{e}$, i.e., by means of the zeros of the characteristic polynomial
\begin{equation}
\mathcal{P}(z)=\vert z\mathbb{I}_{r}-D_{\textbf{X}}\textbf{G}(\textbf{X}_{e})\vert=z^{r}+b_{1}z^{r-1}+\cdots+b_{r-1}z+b_{r},\quad b_{i}\in\mathbb{R}
\end{equation}
Here $\mathbb{I}_{r}$ is the unit $r\times r$ matrix. Due to the structure of the auxiliary equations we will get a block-type structure in the Jacobian $D_{\textbf{X}}\textbf{G}(\textbf{X}_{e})$. The Jacobian of the vector field $\textbf{F}$ evaluated at $\textbf{x}_{e}$ will be one of the blocks in $D_{\textbf{X}}\textbf{G}(\textbf{X}_{e})$. In order to search for possible bifurcations one proceeds in a way analogous to what is demonstrated in the example in Section \ref{Examples}.

\begin{remark} The present formalism can be extended to situations where we have several distributed blocks with distributed delay terms by generalizing the right hand side of 
(\ref{system1}) to a function $\textbf{F}:\mathbb{R}^{D}\times\mathbb{R}^{D}\times \mathbb{R}^{D}\times\cdots\times\mathbb{R}^{D}\rightarrow\mathbb{R}^{D}$. We then get the distributed delay system
\begin{equation}\nonumber
\left.\begin{aligned} 
&\dot{\textbf{x}}(t)=\textbf{F}(\textbf{x}(t),(\Re_{d_{1}}{\textbf{x}})(t), (\Re_{d_{2}}{\textbf{x}})(t),\cdots, (\Re_{d_{m}}{\textbf{x}})(t)),\quad t\in\mathcal{I}_{+}\equiv (0,t^{\ast})&\\
&\textbf{x}(\tau)=\textbf{u}(\tau),\tau\in\mathcal{I}_{-}\equiv (-\infty,0]&
\end{aligned} \right\}
\end{equation}
where the sequence of the delay operators $\{\Re_{d_{j}}\}_{j=1}^{m}$  is of the type (\ref{system2})-(\ref{matrix}). It is also possible to work out the linear chain trick when we have general Lipschitz continuous functions $\textbf{h}_{k}:\mathbb{R}^{D}\rightarrow\mathbb{R}^{D^{\ast}}$ in the delay terms where $D^{\ast}\in\mathbb{N}$, i.e.,
\begin{equation}
\textbf{z}_{k}(t)=\int_{-\infty}^{t}\alpha_{k}(t-s)\textbf{h}_{k}(\textbf{x}(s))ds,\quad k=1,2,\cdots,N\nonumber
\end{equation}
We do not detail the linear chain trick procedure for these generalized cases.
\end{remark}
\begin{remark}
The initial value problem of (\ref{system1}) is locally well-posed. This follows from using the contraction principle for complete metric spaces (Banach's fixed-point theorem). This means that the solution depends continuously on the input data. The set of input data also includes the integral kernels in the delay terms. Notice that this conclusion is not based on the  usage of the linear chain trick. This result means that we can consider the sequence of integral kernels $\{\alpha_{k}\}_{k=1}^{N}$ defined by means of (\ref{system3}) as approximations of the true time histories: The output 
$\textbf{x}$ of the initial value problem (\ref{system1}) with such time histories can (in the space of absolute continuous functions equipped with a supremum norm) be approximated with the output $\textbf{x}_{LCT}$ obtained by means of the LCT-procedure described above. Here we tacitly assume that the integral kernels are absolute integrable on the interval $[0,+\infty)$. See Theorem $1$ in \cite{ponosov2004thew} and Chapter 7.3 in \cite{azbelev2007introduction}. For the sake of completeness we will show continuous dependence on the time histories in the distributed time delay system (\ref{system1}) in \ref{Continuous dependence on the time history} (Theorem \ref{stabilitytheorem}).
\end{remark}

%\section{Examples}\label{Examples} Numerical studies of simple epidemiological  models (SIR- and SEIR-models)  and predator-prey models with distributed time delays.  The purpose of this study is to illustrate the robustness property: Mapping from space of absolute integrable memory functions to the space of absolute continuous solutions of the distributed time delay system is continuous with respect to the metrics.

\section{Concluding remarks and extensions\label{Concluding remarks and extensions}}
\subsection{Main results\label{Main results}}

In the present paper we have worked out the linear chain trick algorithm (LCT) for a class of distributed delay dynamical systems with oscillations in the time history.  We have demonstrated the application of LCT in population dynamics.
%, economic growth models of the Kaldor-Kalecki type and nonlinear optical field with the Raman response incorporated in the description, thus showing that the LCT is applicable to a wide variety of problems involving dynamical models. 
We also show analytically and numerically in our example that our results regarding the stability characteristics of the equilibrium states depend continuously on the time history, in agreement with the fact that the mapping from the Banach-space of absolute integrable memory kernels to the Banach-space of absolute continuous solution of the time delay problem is continuous with respect to the metrics (Theorem \ref{stabilitytheorem} in \ref{Continuous dependence on the time history}). A key point in our population dynamical example is that the memory kernel is given as perturbed Erlang distributions where the perturbation is a finite amplitude temporal oscillation.

\subsection{Possible extensions\label{Possible extensions}} 
Here we will list two possible extensions of the present work. 
%It is of interest to extend the usage of LCT for the Kaldor-Kalecki model (\ref{KaldorKalecki}) - (\ref{normalization}) when the oscillatory effects in the memory kernel is approximated with a superposition of frequency components with different weights.  We conjecture that this assumption mimics better a real macroeconomic situation. Here a parameter fit is needed.\footnote {Private communication with Harald Bergland.} 
\subsubsection{Raman scattering process\label{Raman scattering process}}
In the description of modulated wavetrains in nonlinear optics one opts for the inclusion of the Raman scattering process \cite{blow1989theoretical,laegsgaard2007mode,dudley2010supercontinuum}. This latter effect appears as distributed delay effect in time with oscillations incorporated in the memory function. The actual model assumes the form of an integro-differential equation of the nonlinear Schr\"{o}dinger type. Formally one can apply the LCT-technique to this model. The outcome is a coupled system consisting of a PDE of the nonlinear Schr\"{o}dinger type and rate equations for the auxiliary variables. This system can serve as a starting point for numerical studies of the impact of the Raman process on the optical pulse propagation.\\
\newline
It is of interest to investigate the continuous dependence of the solution on the time history for the Raman response model with respect to appropriate metrics. This is motivated by the fact that the design of time history contains (measurement) uncertainties. We conjecture that this problem can be resolved when investigating the issue of global wellposedness of the actual model. Here we will like to stress that Theorem \ref{stabilitytheorem} in \ref{Continuous dependence on the time history} applied to the finite dimensional dynamical systems of the type (\ref{system1}), and not to the PDE-type of model describing the optical pulse propagation. We will return to the continuity aspect in a future paper. 

\subsubsection{Kaldor-Kalecki economic growth model\label{Kaldor-Kalecki economic growth model}}

In economic growth theory, models of the Kaldor-Kalecki type play a prominent role. These models, which are expressed in terms of coupled nonlinear dynamical systems in the product and the capital stock variables, describe growth cycles in the framework of Keynesian macro-economic theory. In \cite{krawiec2017economic} the actual Kaldor-Kalecki model contains an absolute delay effect in the investment function. One of the outcomes of this extension is the excitation of stable cycles as a consequence of supercritical Hopf bifurcations. A further extension is proposed and discussed in \cite{guerrini2020bifurcations}. Here it is argued that it is impossible to give a precise estimate of the absolute delay time and that a realistic description therefore consists of assuming the delay effect in the investment function to be of the distributed time delay type. The time histories used in \cite{guerrini2020bifurcations} is of the Erlang type with no temporal oscillations incorporated, i.e., of 'weak'- and 'strong' type.  Also in this case supercritical Hopf-bifurcations are detected. In a future work we will incorporate temporal oscillations in the time history, in line with the general framework (\ref{system1})-(\ref{matrix}).

\section*{Data Availability}
No data was used in the present study. The figures are based on simple numeric solutions of the differential equations.

\section*{Conflict of Interests}
No conflicts of interests to declare.

\section*{Acknowledgements} The present work was initiated in September 2022 when J. Wyller (JW) was a Guest Researcher at Department of Applied
Mathematics and Computer Science, Technical University of Denmark. JW would like to express his sincere gratitude to Technical University of Denmark for kind hospitality during the stay. JW gratefully acknowledges the financial support from internal funding scheme at Norwegian University of Life Sciences (project number 1211130114), which financed the international stay at Department of Applied Mathematics and Computer Science, Technical University of Denmark, Denmark. Adam Mielke (AM) was partially funded by Statens Serum Institut, Denmark. The authors want to thank Prof. Arcady Ponosov (Norwegian University of Life Sciences), Prof. Harald Bergland (The Arctic University of Norway) and Prof. Pål Andreas Pedersen (Nord University, Norway) for fruitful and stimulating discussions during the preparation phase of this paper.  We are thankful for the eminent and profound contributions from Snif Mielke.

\appendix

\section{Continuous dependence on the time history}\label{Continuous dependence on the time history}
Our results above only directly apply to kernels that allow the linear chain trick. However, we want to emphasize that the solutions depend continuously on the time histories, which implies a larger degree of generality. We will show this below.\\
\newline
More precisely, we will show that the mapping from the Banach-space of absolute integrable temporal kernels to the Banach-space of absolute continuous solutions of the distributed time delay system (\ref{system1}) is a \emph{continuous mapping with respect to the metrics}. This means that even though two integral kernels point-wise deviate significantly from each other the respective outputs (=the solutions of the system  (\ref{system1})-(\ref{matrix})) can be close to each other.\\
\newline
We proceed as follows: Let $\textbf{x}_{1}$ and $\textbf{x}_{2}$ be solutions of the distributed delay system (\ref{system1}) emanating from the same initial condition, i.e., 
\begin{equation}\label{System1}
\left.\begin{aligned} 
&\dot{\textbf{x}}_{1}(t)=\textbf{F}(\textbf{x}_{1}(t),(\Re_{1}{\textbf{x}_{1}})(t)),\quad t\in\mathcal{I}_{+}\equiv (0,t^{\ast})&\\
&\textbf{x}_{1}(\tau)=\textbf{u}(\tau),\tau\in\mathcal{I}_{-}\equiv (-\infty,0]&
\end{aligned} \right\}
\end{equation}
\begin{equation}\label{System2a}
\left.\begin{aligned} 
&\dot{\textbf{x}}_{2}(t)=\textbf{F}(\textbf{x}_{2}(t),(\Re_{2}{\textbf{x}_{2}})(t)),\quad t\in\mathcal{I}_{+}\equiv (0,t^{\ast})&\\
&\textbf{x}_{2}(\tau)=\textbf{u}(\tau),\tau\in\mathcal{I}_{-}\equiv (-\infty,0]&
\end{aligned} \right\}
\end{equation}
where $t^{\ast}\leq\infty$. Here
\begin{equation}
\label{System2}
\left.\begin{aligned}
&(\Re_{m}{\textbf{x}_{m}})(t)=\sum_{k=1}^{N}c_{k}\textbf{z}_{k}^{(m)}(t)&\\
& \textbf{z}_{k}^{(m)}(t)=\big[\alpha_{k}^{(m)}\otimes\textbf{x}_{d}^{(m)}\big](t)
\equiv\int\limits_{-\infty}^{t}\alpha_{k}^{(m)}(t-s)\textbf{x}_{d}^{(m)}(s)ds\in\mathbb{R}^{D}&\\
&\textbf{x}_{d}^{(m)}\equiv\mathbb{J}\cdot\textbf{x}_{m}=[x_{1}^{(m)}, x_{2}^{(m)},\cdots, x_{d}^{(m)},0,0,\cdots,0]^{T}&
\end{aligned} \right\}
\end{equation}
for $m=1,2$. Here $\mathbb{J}$ is the $D\times D$-matrix defined as (\ref{matrix}). Let $L^{1}[\mathbb{R}_{0}^{+}]$ denote the space of absolute integrable functions on $\mathbb{R}_{0}^{+}$, i.e., 
\begin{equation}
L^{1}[\mathbb{R}_{0}^{+}]=\{f:\mathbb{R}_{0}^{+}\rightarrow\mathbb{C};\|f\|_{1}\equiv\int\limits_{0}^{\infty}\vert f(t)\vert dt<\infty\}\label{absolute1}
\end{equation}
We assume that the sequences of integral kernels $\{\alpha_{k}^{(m)}\}_{k=1}^{N}$ belong to the space $L^{1}[\mathbb{R}_{0}^{+}]$. The norm difference between $\alpha_{k}^{(2)}$ and $\alpha_{k}^{(1)}$ is given as
\begin{equation}
\|\alpha_{k}^{(2)}-\alpha_{k}^{(1)}\|_{1}\equiv\int\limits_{0}^{\infty}\vert\alpha_{k}^{(2)}-\alpha_{k}^{(1)}\vert (t)dt<\infty,\quad k=1,2,\cdots,N\label{absolute2}
\end{equation}
\begin{remark} The integral kernels (\ref{system3})-(\ref{system4}) belong to $L^{1}[\mathbb{R}_{0}^{+}]$.
\end{remark}
We furthermore make use of the norm $\|\cdot\|$ defined as
\begin{equation}
\|\textbf{a}\|\equiv\sum\limits_{k=1}^{D}\vert a_{k}\vert,\quad \textbf{a}=[a_{1}, a_{2},\cdots,a_{D}]\in\mathbb{R}^{D}\label{norm}
\end{equation}
Let $t\in\mathcal{I}_{+}$ and introduce the supremum norm
\begin{equation}
\|\textbf{x}\|_{\infty}(t)\equiv\sup\limits_{s\in[0,t]}\|\textbf{x}(s)\|\label{supremumnorm}
\end{equation}
and the distance function $\Delta:\mathcal{I}_{+}\rightarrow\mathbb{R}_{0}^{+}$ defined as
\begin{equation}
\Delta(t)\equiv\|\textbf{x}_{2}-\textbf{x}_{1}\|_{\infty}(t)\label{distance}
\end{equation}
on the Banach space of absolute continuous functions. This function measures the norm difference between the integral curves of the systems  (\ref{System1}) and  (\ref{System2a}). We have the following result:
\begin{lemma} \label{lemma} The norm of the difference $\Re_{2}{\textbf{x}_{2}}-\Re_{1}{\textbf{x}_{1}}$ satisfies the bounding inequality
\begin{eqnarray}
&\|\Re_{2}{\textbf{x}_{2}}-\Re_{1}{\textbf{x}_{1}}\|(t)\leq\big[\sum\limits_{k=1}^{N}\vert c_{k}\vert \|\alpha_{k}^{(2)}-\alpha_{k}^{(1)}\|_{1}\big]\sup\limits_{s\in (-\infty,t]}\|\textbf{x}_{2}(s)\|&\nonumber\\
&+\big[\sum\limits_{k=1}^{N}\vert c_{k}\vert\|\alpha_{k}^{(1)}\|_{1}\big]\Delta(t)&\label{delaydifferenceestimate}
\end{eqnarray}
for all $t\in\mathcal{I}_{+}$.
\end{lemma}
\begin{proof} In the process of deriving the estimate (\ref{delaydifferenceestimate}) we will make use of the fact that (\ref{System1}) and (\ref{System2a}) have the same initial condition, the absolute integrability of the temporal kernels $\{\alpha_{k}^{(m)}\}_{k=1}^{N}$ (m=1,2), the norm difference (\ref{absolute2}), the fact that $\|\textbf{y}_{d}\|=\|\mathbb{J}\cdot\textbf{y}\|\leq\|\textbf{y}\|$ for any vector $\textbf{y}\in\mathbb{R}^{D}$ and the triangle inequality for integrals. We first obtain
\begin{eqnarray}
&\|\Re_{2}{\textbf{x}_{2}}-\Re_{1}{\textbf{x}_{1}}\|(t)=\|\sum\limits_{k=1}^{N} c_{k}(\alpha_{k}^{(2)}\otimes\textbf{x}_{d}^{(2)}-\alpha_{k}^{(1)}\otimes\textbf{x}_{d}^{(1)})\|(t)&\nonumber\\
&=\|\sum\limits_{k=1}^{N} c_{k}\big[(\alpha_{k}^{(2)}-\alpha_{k}^{(1)})\otimes\textbf{x}_{d}^{(2)}+\alpha_{k}^{(1)}\otimes(\textbf{x}_{d}^{(2)}-\textbf{x}_{d}^{(1)})\big]\|(t)
\label{estimate1}\\
&&\nonumber\\
&\leq\sum\limits_{k=1}^{N}\vert c_{k}\vert\|(\alpha_{k}^{(2)}-\alpha_{k}^{(1)})\otimes\textbf{x}_{d}^{(2)}\|(t)+\sum\limits_{k=1}^{N}\vert c_{k}\vert\|\alpha_{k}^{(1)}\otimes(\textbf{x}_{d}^{(2)}-\textbf{x}_{d}^{(1)})\|(t)&\nonumber
\end{eqnarray}
by means of the suspension trick $a-b=(a-c)+(c-b)$ ($a,b,c\in\mathbb{R}$) and the triangle inequality. Next, we find the estimates
\begin{eqnarray}
&\|(\alpha_{k}^{(2)}-\alpha_{k}^{(1)})\otimes\textbf{x}_{d}^{(2)}\|(t)=\|\int\limits_{-\infty}^{t}(\alpha_{k}^{(2)}-\alpha_{k}^{(1)})(t-s)\textbf{x}_{d}^{(2)}(s)ds\|&\nonumber\\
&&\nonumber\\
&\leq\int\limits_{-\infty}^{t}\vert(\alpha_{k}^{(2)}-\alpha_{k}^{(1)})(t-s)\vert\|\textbf{x}_{d}^{(2)}(s)\|ds&\nonumber\\
&&\nonumber\\
&\leq\int\limits_{-\infty}^{t}\vert(\alpha_{k}^{(2)}-\alpha_{k}^{(1)})(t-s)\vert\|\textbf{x}_{2}(s)\|ds\leq\sup\limits_{s\in (-\infty,t]}\|\textbf{x}_{2}(s)\|\cdot\|\alpha_{k}^{(2)}-\alpha_{k}^{(1)}\|_{1}&\nonumber
\end{eqnarray}
and
\begin{eqnarray}
&\|\alpha_{k}^{(1)}\ast(\textbf{x}_{d}^{(2)}-\textbf{x}_{d}^{(1)})\|(t)=\|\int\limits_{-\infty}^{t}\alpha_{k}^{(1)}(t-s)(\textbf{x}_{d}^{(2)}-\textbf{x}_{d}^{(1)})(s)ds\|&\nonumber\\
&&\nonumber\\
&\leq\int\limits_{-\infty}^{t}\vert\alpha_{k}^{(1)}(t-s)\vert\|(\textbf{x}_{d}^{(2)}-\textbf{x}_{d}^{(1)})(s)\|ds&\nonumber\\
&&\nonumber\\
%&=\overbrace{\int\limits_{-\infty}^{0}\vert\alpha_{k}^{(1)}(-s)\vert\|(\textbf{x}_{d}^{(2)}-\textbf{x}_{d}^{(1)})(s)\|ds}^{\textrm{$=0$ since $\textbf{x}_{2}(s)=\textbf{x}_{1}(s)=\textbf{u}(s)$ for $s\leq 0$}}+\int\limits_{0}^{t}\vert\alpha_{k}^{(1)}(t-s)\vert\|(\textbf{x}_{d}^{(2)}-\textbf{x}_{d}^{(1)})(s)\|ds&\nonumber\\
%&&\nonumber\\
&\leq\int\limits_{-\infty}^{t}\vert\alpha_{k}^{(1)}(t-s)\vert\|(\textbf{x}_{2}-\textbf{x}_{1})(s)\|ds\leq\|\alpha_{k}^{(1)}\|_{1}\cdot
\|\textbf{x}_{2}-\textbf{x}_{1}\|_{\infty}(t)=\|\alpha_{k}^{(1)}\|_{1}\Delta(t)&\nonumber
\end{eqnarray}
We readily obtain (\ref{delaydifferenceestimate}) by combining these estimates  with (\ref{estimate1}).
\end{proof}
The following theorem shows that the mapping from the Banach-space $L^{1}[\mathbb{R}_{0}^{+}]$ to the Banach-space of absolute continuous solutions of the distributed delay system (\ref{system1}) is a continuous map (i.e., continuity with respect to the metrics):
\begin{thm}\label{stabilitytheorem}
Let $T\in\mathcal{I}_{+}$ with $T<\infty$. The distance function $\Delta$ satisfies the stability estimate
\begin{equation}
\Delta(T)\leq LT\sup\limits_{t\in(-\infty,T]}(\|\textbf{x}_{2}\|(t))\bigg[\sum\limits_{k=1}^{N}\vert c_{k}\vert\cdot\|\alpha_{k}^{(2)}-\alpha_{k}^{(1)}\|_{1}\bigg]\exp\big(\Omega_{N}LT\big)\label{stabilityestimate}
\end{equation}
where
\begin{equation}
\Omega_{N}\equiv\big[1+\sum\limits_{k=1}^{N}\vert c_{k}\vert\|\alpha_{k}^{(1)}\|_{1}\big]\nonumber
\end{equation}
and $L$ is given as 
\begin{equation}
L\equiv\sup\limits_{(\textbf{x},\textbf{y})\in U}\big[\|D\textbf{F}\|_{W^{1,\infty}(U)}\big]\label{Lipschitzconstant}
\end{equation}
Here $D\textbf{F}$ is the Jacobian of the differential of $\textbf{F}$ and $L$ is the operator norm. Here $W^{1,\infty}(U)$ denotes the Sobolev space where the elements and their distributional derivatives are $L^{\infty}$-functions on an open, connected, convex Euclidean subset $U$ of the product space $\mathbb{R}^{D}\times\mathbb{R}^{D}$.
\end{thm}
\begin{proof} 
We first derive the integral equations 
\begin{equation}
\textbf{x}_{2}(T)=\textbf{u}(0)+\int\limits_{0}^{T}\textbf{F}(\textbf{x}_{2}(t),(\Re_{2}{\textbf{x}_{2}})(t))dt,
\quad\textbf{x}_{1}(T)=\textbf{u}(0)+\int\limits_{0}^{T}\textbf{F}(\textbf{x}_{1}(t),(\Re_{1}{\textbf{x}_{1}})(t))dt\nonumber
\end{equation}
by formal integration of (\ref{System1}) and (\ref{System2}). Hence, we get the bound
\begin{equation}
\|\textbf{x}_{2}-\textbf{x}_{2}\|(T)\leq\int\limits_{0}^{T}\|\textbf{F}(\textbf{x}_{2},\Re_{2}{\textbf{x}_{2}})-\textbf{F}(\textbf{x}_{1},\Re_{1}{\textbf{x}_{1}})\|(t)dt\label{bound1}
\end{equation}
by the triangle inequality for integrals. 
%The suspension trick $a-b=(a-c)+(c-b)$ and the triangle inequality yield
%\begin{equation}
%\|\textbf{F}(\textbf{x}_{2},\Re_{2}{\textbf{x}_{2}})-\textbf{F}(\textbf{x}_{1},\Re_{1}{\textbf{x}_{1}})\|
%\leq\|\textbf{F}(\textbf{x}_{2},\Re_{2}{\textbf{x}_{2}})-\textbf{F}(\textbf{x}_{2},\Re_{1}{\textbf{x}_{1}})\|+
%\|\textbf{F}(\textbf{x}_{2},\Re_{1}{\textbf{x}_{1}})-\textbf{F}(\textbf{x}_{1},\Re_{1}{\textbf{x}_{1}})\|\label{bound2}
%\end{equation}
%We will achieve this by assuming the function $\textbf{F}:U\rightarrow\mathbb{R}^{D}$ defining the right hans side of our system to be a smooth function. Here $U$ is an open, connected, convex Euclidean subset of the product space $\mathbb{R}^{D}\times\mathbb{R}^{D}$ and it is assumed that $(\textbf{u}(0),(\Re{\textbf{u}})(0)),(\textbf{x},\Re{\textbf{x}})\in U$.
Since by assumption $\textbf{F}$ is continuously differentiable on the open, connected, convex Euclidean subset $U$ of the product space $\mathbb{R}^{D}\times\mathbb{R}^{D}$, $\textbf{F}$ will satisfy the Lipschitz condition
\begin{equation}
\|\textbf{F}(\textbf{x}_{2},\textbf{y}_{2})-\textbf{F}(\textbf{x}_{1},\textbf{y}_{1})\|\leq L\|\textbf{x}_{2}-\textbf{x}_{1}\|+L\|\textbf{y}_{2}-\textbf{y}_{1}\|\label{Lipschitzcondition}
\end{equation}
for all $(\textbf{x}_{2},\textbf{y}_{2}),(\textbf{x}_{1},\textbf{y}_{1})\in U$, where the Lipschitz constant $L$ is given as (\ref{Lipschitzconstant}).
%\newpage
%\begin{eqnarray}
%&\|\textbf{F}(\textbf{x}_{2},\Re_{2}{\textbf{x}_{2}})-\textbf{F}(\textbf{x}_{2},\Re_{1}{\textbf{x}_{1}})\|\leq L_{2}\|\Re_{2}{\textbf{x}_{2}}-\Re_{1}{\textbf{x}_{1}}\|&\nonumber\\
%&&\nonumber\\
%&\|\textbf{F}(\textbf{x}_{2},\Re_{1}{\textbf{x}_{1}})-\textbf{F}(\textbf{x}_{1},\Re_{1}{\textbf{x}_{1}})\|\leq L_{1}\|\textbf{x}_{2}-\textbf{x}_{1}\|&\nonumber
%\end{eqnarray}
%where $L_{1}$ is the matrix norm of the Jacobian with respect to the first variable in $\textbf{F}$ for fixed $\Re_{1}{\textbf{x}_{1}}$ whereas $L_{2}$ is the matrix norm of the Jacobian with respect to the second variable for fixed $\textbf{x}_{2}$.  By introducing $L=\max(L_{1},L_{2} )$, we find that
%\begin{equation}
%\|\textbf{F}(\textbf{x}_{2},\Re_{2}{\textbf{x}_{2}})-\textbf{F}(\textbf{x}_{1},\Re_{1}{\textbf{x}_{1}})\|\leq L\|\textbf{x}_{2}-\textbf{x}_{1}\|+L\|\Re_{2}{\textbf{x}_{2}}-\Re_{1}{\textbf{x}_{1}}\|\label{System3}
%\end{equation}
%Now, by taking the Lipschitz condition (\ref{System3}) and the triangle inequality into account, we find that
%\begin{equation}
%\|\textbf{x}_{2}-\textbf{x}_{1}\|(T)\leq L\int\limits_{0}^{T}
%\|\Re_{2}{\textbf{x}_{2}}-\Re_{1}{\textbf{x}_{1}}\|(t)dt+L\int\limits_{0}^{T}\|\textbf{x}_{2}-\textbf{x}_{1}\|(t)dt
%\end{equation}
By exploiting (\ref{Lipschitzcondition}) (with $\textbf{y}_{2}=\Re_{2}{\textbf{x}_{2}}$ and $\textbf{y}_{1}=\Re_{1}{\textbf{x}_{1}}$) and Lemma \ref{lemma} we get 
\begin{eqnarray}
&\|\textbf{x}_{2}-\textbf{x}_{1}\|(T)\leq LT\bigg[\sum\limits_{k=1}^{N}\vert c_{k}\vert\cdot\|\alpha_{k}^{(2)}-\alpha_{k}^{(1)}\|_{1}\bigg]
\sup\limits_{t\in(-\infty,T]}(\|\textbf{x}_{2}\|(t))&\nonumber\\
&&\nonumber\\
&+L\bigg[1+\sum\limits_{k=1}^{N}\vert c_{k}\cdot\vert\|\alpha_{k}^{(1)}\|_{1}\bigg]\int\limits_{0}^{T}\Delta
(t)dt&\nonumber
\end{eqnarray}
from which it follows that the distance function $\Delta$ satisfies the integral inequality
\begin{eqnarray}
&\Delta(T)\leq LT\bigg[\sum\limits_{k=1}^{N}\vert c_{k}\vert\cdot\|\alpha_{k}^{(2)}-\alpha_{k}^{(1)}\|_{1}\bigg]
\sup\limits_{t\in(-\infty,T]}(\|\textbf{x}_{2}\|(t))
&\nonumber\\
&&\nonumber\\
&+L\bigg[1+\sum\limits_{k=1}^{N}\vert c_{k}\vert\cdot\|\alpha_{k}^{(1)}\|_{1}\bigg]\int\limits_{0}^{T}\Delta
(t)dt&\nonumber
\end{eqnarray}
By using the integral form of Gr\"{o}nwalls lemma, we end up with \eqref{stabilityestimate}. This completes the proof.
\end{proof}
\begin{remark}The integral form of Gr\"{o}nwalls lemma reads: Let $I$ denote an interval of the real line of the form $[a, \infty)$, $[a, b]$ or $[a, b)$ with $a < b$. Let $\alpha$, $\beta$ and $u$ be real-valued functions defined on $I$, with the following properties:.
\begin{itemize}
\item The negative part of $\alpha$ is integrable on every closed and bounded sub-interval of $I$ and $\alpha$ is a non-decreasing function.
\item $u$ and $\beta$ are continuous on $I$.
\end{itemize}
Then, if $\beta$ is non-negative and $u$ satisfies the integral inequality
\begin{equation}
u(T)\leq\alpha(T)+\int\limits_{a}^{T}\beta(t)u(t)dt,\quad\forall T\in I\nonumber
\end{equation}
we have
\begin{equation}
u(T)\leq\alpha(T)\exp\bigg[\int\limits_{a}^{T}\beta(t)dt\bigg]\nonumber
\end{equation}
See Bellmann \cite{bellman1943}. In our case
\begin{eqnarray}
&\alpha(T)=LT\sup\limits_{t\in(-\infty,T]}(\|\textbf{x}_{2}\|(t))\bigg[\sum\limits_{k=1}^{N}\vert c_{k}\vert\cdot\|\alpha_{k}^{(2)}-\alpha_{k}^{(1)}\|_{1}\bigg]&\nonumber\\
&&\nonumber\\
&\beta(t)=L\bigg[1+\sum\limits_{k=1}^{N}\vert c_{k}\vert\cdot\|\alpha_{k}^{(1)}\|_{1}\bigg],\quad u(t)=\Delta(t)=\|\textbf{x}_{2}-\textbf{x}_{1}\|_{\infty}(t)&\nonumber
\end{eqnarray}
\end{remark}
%\begin{remark}
%We notice that the conditions for using the continuous dependence result in Subsection \ref{Continuous dependence on the time history} (Theorem \ref{stabilitytheorem}) are fulfilled for %the SIR - model (\ref{Eq:SIR-delay}) - (\ref{Eq:kernelI}) and the SEIR - model (\ref{Eq:SEIR-model}) - (\ref{Eq:betaInt}). 
%\end{remark}
\begin{remark} Theorem \ref{stabilitytheorem} holds true when the coefficients $c_{1},c_{2},\cdots, c_{N}$ in (\ref{system2}) are continuous on the closed bounded interval $[0,T]$. Here $T<\infty$.
\end{remark}

\section{The characteristic polynomial (\ref{characteristicpolynomiallogistic})}\label{Appendix B}
%Here we collect the expressions for the coefficients of the characteristic polynomial (\ref{characteristicpolynomiallogistic}).
%\subsection{\label{Appendix B1}}
The coefficients of the characteristic polynomial (\ref{characteristicpolynomiallogistic}) are given as
\begin{subequations}
\label{populationdynamicscoeff}
\begin{eqnarray}
&a_{1}=-tr\big[D_{\textbf{x}}\textbf{F}(P_{e})\big]=6\sigma,\quad a_{2}=\sigma^{2}[2\theta+15]&\label{populationdynamicscoeffa}\\
&&\nonumber\\
&a_{3}=\sigma^{3}[20+8\theta+\varepsilon\mu+\mu]&\label{populationdynamicscoeffb}\\
&&\nonumber\\
&a_{4}=\sigma^{4}[\theta^{2}+12\theta+15+4\varepsilon\mu+4\mu]&\label{populationdynamicscoeffc}\\
&&\nonumber\\
&a_{5}=\sigma^{5}[2\theta^{2}+8\theta+6+2\mu\theta+6\mu-\varepsilon \mu\theta+6\varepsilon\mu]&\label{populationdynamicscoeffd}\\
&&\nonumber\\
&a_{6}=\sigma^{6}[\theta^{2}+2\theta+1+4\mu\theta+4\mu-2\varepsilon \mu\theta+4\varepsilon\mu]&\label{populationdynamicscoeffe}\\
&&\nonumber\\
&a_{7}=-\vert D_{\textbf{x}}\textbf{F}(P_{e})\vert =-\sigma^{7}\mu\big[\varepsilon(\theta-1)-(\theta+1)^{2}\big]&\label{populationdynamicscoefff}
\end{eqnarray}
\end{subequations}
Here the parameter $\theta$ is defined by means of (\ref{equilibriumlogisticd}) whereas $\mu_{0}$ and $\mu_{\varepsilon}$ are defined as
\begin{equation}
\theta\equiv\Omega^{2}\sigma^{-2},\quad\mu_{0}\equiv r\sigma^{-1},\quad\mu\equiv \mu_{0}\cdot(1+\theta)^{2}[(1+\theta)^{2}-\varepsilon(\theta-1)]^{-1}\label{parametersthetamu}
\end{equation}
By introducing the scaling $z\rightarrow \sigma z$, we find that $\mathcal{P}_{7}$ obeys the scaling property
\begin{equation}
\mathcal{P}_{7}(\sigma z;\textbf{a})=\sigma^{7}\mathcal{P}_{7}(z;\textbf{b}),\quad \textbf{a}=\mathbb{S}\textbf{b}\label{scaling}
\end{equation}
where 
\begin{equation}
\textbf{a}=[a_{1},a_{2},a_{3},a_{4},a_{5},a_{6},a_{7}]^{t}\in\mathbb{R}^{7}\nonumber
\end{equation}
The components $a_{j}$ ($j=1,2,\cdots,7$) are given by (\ref{populationdynamicscoeff}). The coordinate vector $\textbf{b}$ is given as
\begin{equation}
\textbf{b}=[b_{1},b_{2},b_{3},b_{4},b_{5},b_{6},b_{7}]^{t}\in\mathbb{R}^{7}\nonumber \end{equation}
with components $b_{j}$ ($j=1,2,\cdots,7$) given as
\begin{eqnarray}
&b_{1}=6,\quad b_{2}=2\theta+15,\quad b_{3}=20+8\theta+\varepsilon\mu+\mu&\nonumber\\
&&\nonumber\\
&b_{4}=\theta^{2}+12\theta+15+4\varepsilon\mu+4\mu,\quad b_{5}=2\theta^{2}+8\theta+6+2\mu\theta+6\mu-\varepsilon\mu\theta+6\varepsilon\mu&\nonumber\\
&&\nonumber\\
&b_{6}=\theta^{2}+2\theta+1+4\mu\theta+4\mu-2\varepsilon\mu\theta+4\varepsilon\mu,\quad b_{7}=-\mu\big[\varepsilon(\theta-1)-(\theta+1)^{2}\big]&\nonumber
\end{eqnarray}
and $\mathbb{S}$ as the diagonal matrix
\begin{equation}
\mathbb{S}
=\left[
  \begin{array}{cccccccc}
\sigma& 0& 0& 0& 0&0&0 \\
0&\sigma^{2}&0&0&0&0&0\\
0&0&\sigma^{3}&0&0&0&0\\
0&0&0&\sigma^{4}&0&0&0\\
0&0&0&0&\sigma^{5}&0&0\\
0&0&0&0&0&\sigma^{6}&0\\
0&0&0&0&0&0&\sigma^{7}
    \end{array}
\right]\nonumber
\end{equation}
The scaling property (\ref{scaling}) shows that the stability assessment can be carried out with the coefficients in the characteristic polynomial $\mathcal{P}_{7}$ given by the $b_{j}$ - parameters. 
In this case the Routh-Hurwitz matrices for the polynomial $\mathcal{P}_{7}$ are given as
\begin{eqnarray}
&\mathbf{D}_{1}(\textbf{b})=[b_{1}],\quad \mathbf{D}_{2}(\textbf{b})=\left[
  \begin{array}{cccccccc}
b_{1}& b_{3}\\
1& b_{2}
    \end{array}
\right],\quad \mathbf{D}_{3}(\textbf{b})=\left[
  \begin{array}{cccccccc}
b_{1}& b_{3}&b_{5}\\
1& b_{2}& b_{4}\\
0& b_{1}&b_{3}
    \end{array}
\right]&\nonumber\\
& \mathbf{D}_{4}(\textbf{b})=\left[
  \begin{array}{cccccccc}
b_{1}& b_{3}&b_{5}&b_{7}\\
1& b_{2}& b_{4}& b_{6}\\
0& b_{1}&b_{3}& b_{5}\\
0& 1& b_{2} & b_{4}
    \end{array}
\right],\quad  \mathbf{D}_{5}(\textbf{b})=\left[
  \begin{array}{cccccccc}
b_{1}& b_{3}&b_{5}&b_{7}&0\\
1& b_{2}& b_{4}&b_{6}&0\\
0& b_{1}&b_{3}& b_{5}&b_{7}\\
0& 1& b_{2} & b_{4}&b_{6}\\
0& 0&b_{1}&b_{3}&b_{5}
    \end{array}
\right]&\nonumber\\
&\mathbf{D}_{6}(\textbf{b})=\left[
  \begin{array}{cccccccc}
b_{1}& b_{3}&b_{5}&b_{7}&0&0\\
1& b_{2}& b_{4}&b_{6}&0&0\\
0& b_{1}&b_{3}& b_{5}&b_{7}&0\\
0& 1& b_{2} & b_{4}&b_{6}&0\\
0& 0&b_{1}&b_{3}&b_{5}& b_{7}\\
0 & 0&1&b_{2}&b_{4}&b_{6}  \end{array}
\right],\quad 
\mathbf{D}_{7}(\textbf{b})=\left[
  \begin{array}{cccccccc}
b_{1}& b_{3}&b_{5}&b_{7}&0&0&0\\
1& b_{2}& b_{4}&b_{6}&0&0&0\\
0& b_{1}&b_{3}& b_{5}&b_{7}&0&0\\
0& 1& b_{2} & b_{4}&b_{6}&0&0\\
0& 0&b_{1}&b_{3}&b_{5}& b_{7}&0\\
0 & 0&1&b_{2}&b_{4}&b_{6} &0\\
0&0&0&b_{1}&b_{3}&b_{5}& b_{7}
 \end{array}
\right]&\nonumber
\end{eqnarray}
By using (\ref{populationdynamicscoeff}) one can show that the Routh-Hurwitz determinants $\vert \mathbf{D}_{j}\vert\equiv\vert \mathbf{D}_{j}(\textbf{a})\vert$ of (\ref{characteristicpolynomiallogistic}) scale as 
\begin{equation}
\vert\mathbf{D}_{j}(\textbf{a})\vert=\sigma^{\frac{1}{2}j(j+1)}\vert\mathbf{D}_{j}(\textbf{b})\vert,\quad j=1,2,\cdots,7\nonumber
\end{equation}
Therefore we can base the analysis of the stability and bifurcation analysis on the properties of the Routh-Hurwitz determinants $\vert\mathbf{D}_{j}(\textbf{b})\vert)$ instead of $\vert \mathbf{D}_{j}(\textbf{a})\vert$.\\
\newline
%We readily observe that the Routh - Hurwitz determinants $\vert\mathbf{D_{j}}\vert$ ($j=1,2,\cdots,7$) are smooth functions of $(\theta,\varepsilon,\mu_{0})$.
Noticing that $\mu_{\varepsilon}=\mu_{0}=2$ for $\varepsilon=0$ and $\sigma=\frac{1}{2}r$, direct computation
%(by exploiting the Linear Algebra Calculator capabilities in Symbolab: https://www.symbolab.com/solver/linear-algebra-calculator)
yields the expressions
\begin{eqnarray}
 &\vert \mathbf{D_{1}}(\textbf{b})\vert=6>0,\quad \vert \mathbf{D_{2}}(\textbf{b})\vert=4\theta+68>0 &\nonumber\\
 &&\nonumber\\
 &\vert\mathbf{D_{3}}(\textbf{b})\vert=8\theta^{2}+272\theta+776>0&\nonumber\\
&&\nonumber\\
&\vert\mathbf{D_{4}}(\textbf{b})\vert=64\theta^{3}+1216\theta^{2}+6336\theta+5184>0&\nonumber\\
&&\nonumber\\
&\vert\mathbf{D_{5}}(\textbf{b})\vert=128\theta^{5}+2688\theta^{4}+17664\theta^{3}+38144\theta^{2}+33408\theta+10368>0&\nonumber\\
&&\nonumber\\
&\vert\mathbf{D_{6}}(\textbf{b})\vert=0,\quad b_{7}=2(\theta+1)^{2}>0,\quad \vert\mathbf{D_{7}}(\textbf{b})\vert=b_{7}\vert\mathbf{D_{6}}(\textbf{b})\vert=0&\nonumber
\end{eqnarray}
for the Routh-Hurwitz determinants $\vert\mathbf{D}_{j}(\textbf{b})\vert$ ($j=1,2,\cdots, 7$) evaluated at the points $(\theta,\varepsilon,\mu_{0})=(\theta,0,2)$. For $\Omega$ and $r>0$ given, the transformation (\ref{parametersthetamu}) defines a smooth mapping $(\sigma,\varepsilon)\rightarrow (\theta,\varepsilon,\mu_{0})$ for $\sigma>0$. Since the Routh-Hurwitz determinants $\vert\mathbf{D}_{j}(\textbf{b})\vert$ ($j=1,2,\cdots,7$) are smooth functions of $(\theta,\varepsilon,\mu_{0})$, we conclude that the composite mapping $(\sigma,\varepsilon)\rightarrow (\theta,\varepsilon,\mu_{0})\rightarrow\vert\mathbf{D}_{j}(\textbf{b})\vert$ ($j=1,2,\cdots,7$) is a smooth mapping for $\sigma>0$, as well. Moreover, there is an open neighborhood about $(\frac{1}{2}r,0)$ where $\vert\mathbf{D}_{j}(\textbf{b})\vert>0$ ($j=1,2,\cdots, 5$). Hence we get the local extension of the Hopf-bifurcation result reported in Section \ref{Examples}.

\section*{Conflict of Interest}
Not applicable.

\bibliographystyle{elsarticle-num}
%\bibliography{LiteraturlisteHB2018}
\bibliography{DistributedDelaySystemsOscillations}

\end{document}